\documentclass[12pt]{amsart}
\usepackage{latexsym}
\usepackage{amsmath,amsthm,amsfonts,amscd,eucal,amssymb}
\usepackage{epsfig}
\usepackage{enumerate}
\usepackage{xcolor}
\usepackage{tikz}

\newtheorem{thm}{Theorem}[section]
\newtheorem{cor}[thm]{Corollary}

\newtheorem{lem}[thm]{Lemma}

\newtheorem{prop}[thm]{Proposition}

\theoremstyle{definition}
\newtheorem{dfn}[thm]{Definition}
\theoremstyle{remark}
\newtheorem{rem}[thm]{Remark}
\newtheorem{ex}[thm]{Example}
\newtheorem*{ack}{Acknowledgments}

\numberwithin{equation}{section}

\def\ca{{\mathcal A}}
\def\cb{{\mathcal B}}

\def\ce{{\mathcal E}}

\def\ch{{\mathcal H}}

\def\cl{{\mathcal L}}

\def\car{{\mathcal R}}

\def\bc{{\mathbb C}}

\newcommand{\bn}{\mathbb N}
\def\br{{\mathbb R}}

\renewcommand{\a}{\alpha}
\def\b{\beta}
        
\def\d{\delta}        
\def\eps{\varepsilon}

\def\th{\vartheta}

\def\l{\lambda}

\def\x{\xi}

\def\r{\rho}
\def\s{\sigma}       
    
\def\t{\tau}

\def\f{\varphi}

\def\o{\omega}        

\DeclareMathOperator{\tr}{tr}
\DeclareMathOperator{\Res}{Res}

\DeclareMathOperator{\UHF}{UHF}

\DeclareMathOperator{\osc}{Osc}

\DeclareMathOperator{\diag}{diag}

\begin{document}

\title{A noncommutative Sierpinski Gasket}
\author{Fabio Cipriani}
\address{(F.C.) Politecnico di Milano, Dipartimento di Matematica, piazza Leonardo da Vinci 32, 20133 Milano, Italy.} \email{fabio.cipriani@polimi.it}
\author{Daniele Guido}
\address{(D.G.) Dipartimento di Matematica, Universit\`a di Roma ``Tor Vergata'', I--00133 Roma, Italy.} \email{guido@mat.uniroma2.it}
\author{Tommaso Isola}
\address{(T.I.) Dipartimento di Matematica, Universit\`a di Roma ``Tor Vergata'', I--00133 Roma, Italy} \email{isola@mat.uniroma2.it}
\author{Jean-Luc Sauvageot}
\address{(J.-L.S.) Institut de Math\'ematiques de Jussieu - Paris rive gauche, CNRS -  Universit\'e de Paris - Universit\'e Paris Sorbonne, F 75205 PARIS Cedex 13, France.}
\email{jean-luc.sauvageot@imj-prg.fr}


\date{}

\begin{abstract}
A quantized version of the Sierpinski gasket is proposed, on purely topological grounds, as a $C^*$-algebra $\ca_\infty$ with a suitable form of self-similarity. Several properties of $\ca_\infty$ are studied, in particular its nuclearity, the structure of ideals as well as the description of irreducible representations and extremal traces. A harmonic structure is introduced, giving rise to a self-similar Dirichlet form $\mathcal{E}$. A spectral triple is also constructed, extending one already known for the classical gasket, from which $\mathcal{E}$ can be reconstructed. Moreover we show that $\ca_\infty$ is a compact quantum metric space.
\end{abstract}
\maketitle

\section{Introduction}

In this note we describe an example of a noncommutative fractal, more precisely a noncommutative analogue of a Sierpinski gasket.

By the particular structure of the Sierpinski gasket $K$, the self-similarity relation can be reformulated as a family of embeddings of finite-dimensional algebras, which give rise to a suitable infinite-dimensional $C^*$-algebra which turns out to consist of the continuous functions on $K$. Such procedure may be quantized, producing a noncommutative $C^*$-algebra $\ca_\infty$ with trivial centre which we call the algebra of the noncommutative Sierpinski gasket.

Such $C^*$-algebra contains $C(K)$ as a maximal abelian sub-algebra, and is contained in the UHF algebra $\UHF(3^\infty)=M_3(\bc)^{\otimes \bn}$. The irreducible representations of $\ca_\infty$ are completely described, and are either finite-dimensional or faithful. In the latter case they are restrictions of representations of the UHF algebra.

Then we show that the noncommutative Sierpinski gasket admits a harmonic structure, giving rise to a self-similar closed Dirichlet form on $\ca_\infty$ which restricts to the standard Dirichlet form on $C(K)$.

Finally, we build a spectral triple on $\ca_\infty$, which is strictly related to the spectral triple considered in \cite{GuIs16}. This triple has metric dimension equal to the Hausdorff dimension of the gasket $d=\frac{\log 3}{\log 2}$, and the Connes' trace formula gives back the restriction to $\ca_\infty$ of the unique trace on the UHF algebra, and coincides with the Hausdorff integral when computed on the continuous functions on $K$. The commutator with the Dirac operator produces a Lip-norm in the sense of Rieffel \cite{Rief}, hence $\ca_\infty$ is a compact quantum metric space. The Dirichlet form on $\ca_\infty$ may also be recovered through the spectral triple, according to a singular trace formula already used in \cite{CGIS1} and \cite{GuIs16}.

The Sierpinski gasket $K\subset\mathbb{R}^2$ is among the most studied self-similar fractal sets: it can be reconstructed as a whole from any arbitrary small piece of it. It is the only compact subset of $\br^2$ satisfying
\[
K=w_1(K)\cup w_2(K)\cup w_3(K)
\]
where the  $w_i$'s are the similitudes of scaling parameter $1/2$ fixing the vertices of an equilateral triangle in the plane. Its geometry is governed by the {\it pseudo semigroup} of local homeomorphisms generated by the $w_i$. It has been introduced by Sierpinski \cite{Sierpinski} as a somewhat paradoxical ``curve" whose points ramify. \\
During the 1970s and 1980s physicists \cite{RT,R} became interested in studying the Schr\"odinger equation on $K$ due to the appearance of {\it localized states}, thus suggesting that fractals like $K$ lie in between periodic crystalline and disordered structures (modeled by random potentials). In these works Laplacian-like operators and their spectra were only defined and analyzed as limits of Laplacians on approximating graphs (decimation method) (see however also \cite{BarKi}, \cite{BST}).\\

\noindent On the probabilistic side, $K$ started to be investigated in \cite{Ku1}, \cite{BarPer} and \cite{Gold} where self-similar diffusions where constructed. These processes, though  playing a role similar to the Brownian motion in the plane, exhibit unusual peculiarities summarized by different non integral Hausdorff, metric and random-walk dimensions.\\
However, since no open set in $K$ is homeomorphic to a Euclidean disk, the study of $K$ is challenging in several directions. Difficulties arise already at the topological level as the fundamental group $\pi_1(K)$ is locally compact but not discrete, in a suitable topology (\cite{ADTW}).\\

\noindent Hence $K$ does not admit a manifold structure so that differentiability of functions has to be introduced unconventionally (\cite{CiSa}, \cite{Cip}). This can be done by the analogue $\mathcal{E}$ of the Dirichlet integral on $K$, constructed by Kusuoka \cite{Ku} using probabilistic methods, and fully understood in the late 1980s within Kigami's theory of harmonic structures (\cite{Kiga}). This framework, allowing the use of the theory of Dirichlet forms to reconcile the approaches followed by probabilists and physicists, also provided a firm ground to study spectra of Laplacian-like operators (\cite{FS}, \cite{KiLa1, KiLa2}). The object of the theory is the construction and the study of self-similar Dirichlet forms \cite{FOT}. These are quadratic forms $\mathcal{E}$ on the algebra $C(K)$ of continuous functions, which on the one hand satisfy the fundamental contraction property of the Dirichlet integrals in Euclidean spaces on which the potential theory is based
\[
\mathcal{E}[u\wedge 1]\le\mathcal{E}[u]
\]
and, on the other hand, match the self-similar nature of $K$ in the sense that
\[
\mathcal{E}[u]=r_1^{-1}\cdot \mathcal{E}[u\circ w_1]+r_2^{-1}\cdot\mathcal{E}[u\circ w_2]+r_3^{-1}\cdot\mathcal{E}[u\circ w_3]
\]
for some fixed set of ``energy ratios" $r_1,r_2,r_3>0$. Starting from the differential structure underlying any regular Dirichlet form \cite{CiSa}, the potential theory on $K$ is then developed in such a way that analogues of the Hodge and de Rham theorems are available \cite{CGIS1} (see also \cite{SZSY}). Here a first peculiarity of $K$, compared with smooth compact Riemannian manifolds, is that the space of harmonic 1-forms is infinite dimensional. A second one is related to the fact that $K$ is not semilocally simply connected so that it does not admit a universal (simply connected) covering space and potentials of locally exact 1-forms have to be realized on the Uniform Universal Abelian Covering (\cite{PW}, \cite{CGIS1}). Finally, volume and energy distribution on the Sierpinski gasket can be described in the framework of Noncommutative Geometry by a suitable Dirac operator \cite{CGIS2} (see also \cite{CIL}, \cite{CIS}). In agreement with the discovery by Kusuoka and Hino that energy measures on $K$ associated to $\mathcal{E}$ are singular with respect to any self-similar Bernoulli measure (\cite{Ku}, \cite{Hino}), \cite{HiNa}, it results that the (non integer) dimensions of $K$ relative to the volume and energy distributions associated to $D$ differ (in contrast to the Riemannian situation).
\vskip0.2truecm

\noindent We now describe our idea of the construction of a noncommutative fractal modeled on the Sierpinski gasket.
\vskip0.2truecm\noindent
The self-similarity equation for a fractal means that the union of a finite number of (rescaled) copies of the space coincides with the space itself. Up to more or less trivial cases, the union is not disjoint, and, for the gasket, only involves finitely many vertices, so that the fractal can be completely reconstructed  by the projections from the disjoint union of vertices onto the non-disjoint one. From a functional point of view, this amounts to an embedding of finite-dimensional algebras, namely  the map $\psi_{n+1}^*:\bc(V_{n+1})\hookrightarrow\bc^3\otimes \bc(V_{n})$ (the space of complex valued functions on $V_n$). The family $\{\bc(V_n)\}_{n\geq0}$ is not an inductive family, however any of the algebras, viewed as an algebra of  cylindrical functions, naturally embeds  into the algebra of functions on the infinite product $\Sigma=\{0,1,2\}^\bn$ endowed with the product topology.
It turns out that the algebra of continuous functions on $K$ is a suitable limit of the $\bc(V_n)$'s seen as sub-algebras of $C(\Sigma)$.

\noindent As observed before, the maps $\psi_{n+1}^*$ are induced by surjections $\psi_{n+1}:V_n\sqcup V_n\sqcup V_n\to V_{n+1}$. Such maps are injective up to their restriction to the external vertices in $V_n$, the non-injectivity pattern being the same as that for the projection map $\psi_1:V_0\sqcup V_0\sqcup V_0\to V_{1}$. This means that the algebras $\bc(V_n)$ and the maps $\psi_n^*$, and hence the algebra $C(K)$, can be inductively reconstructed starting from $V_0$, $V_1$ and $\psi_1$.

\noindent The quantisation procedure consists in replacing the algebras $\bc(V_{n})$ with (generally) noncommutative finite-dimensional $C^*$-algebras $\ca_n$ endowed with embeddings
(which we denote with the same symbol) $\psi_{n+1}^*:\ca_{n+1}\hookrightarrow M_3(\bc)\otimes \ca_n$ (see Definition \ref{IndDef}). As in the commutative case, the algebras $\ca_n$ and the embeddings $\psi_n^*$ are determined inductively by $\ca_0$, $\ca_1$ and $\psi_1^*$. The base of the induction consists in posing $\ca_0=\bc(V_0)$, $\ca_1=\bc(V_1)$ and $\psi_1^*$ equal to its commutative counterpart. However, the replacement of $\bc^3$ with $M_3(\bc)$ allows $\ca_n$, $n\geq2$, to be noncommutative.
A visualization of the difference between the embeddings $\bc(V_{2})\hookrightarrow\bc^3\otimes \bc(V_{1})$ and $\ca_{2}\hookrightarrow M_3(\bc)\otimes \ca_1$ is shown by the Bratteli diagrams in Fig. \ref{fig:diagrams}.

\noindent Because of the inclusions $\psi_n^*$, all algebras $\ca_n$ naturally embed in $\UHF(3^\infty)=M_3(\bc)^{\otimes \bn}$. As in the classical case, the sequence $\{\ca_n\}_{n\geq0}$ is not inductive, however norm-convergent sections $n \in\bn \mapsto a_n\in\ca_n$ form a $C^*$-subalgebra of $\UHF(3^\infty)$ (see Theorem \ref{AinfAlgebra}), which we call the quantum gasket $\ca_\infty$.

\noindent Let us notice that the construction of $C(K)$ outlined above is purely topological, since it  uses neither the metric information nor the harmonic structure of the gasket, but only the topology of $\Sigma$. The same is true for the quantum gasket $\ca_\infty$, where the topological information is given by the embeddings of the $\ca_n$'s in $\UHF(3^\infty)$.


\noindent Once the $C^*$-algebra of quantum gasket is defined, we study its representations. A key step here is the existence of restriction maps $\r_n:\ca_\infty\to\ca_n$ which are $C^*$-epimorphisms. Such maps, which are the counterpart of the natural embeddings $V_n\to K$, turn out to describe all irreducible finite-dimensional representations of $\ca_\infty$.

\noindent In order to define the maps $\r_n$ we introduce a non-standard numbering of the vertices of the gasket, which gives rise to a labelling of the generators of $\ca_n$ expressed as tensor products of $3\times 3$ matrices. The map from $M_3(\bc)^{\otimes (n+1)}\to M_3(\bc)^{\otimes n}$, which gives 0 if the last tensor factor is not the matrix unit $e_{22}$ and erase the last tensor factor otherwise, turns out to restrict to a morphism $\ca_{n+1}\to\ca_n$. Combining such morphisms we produce the maps $\r_n$.
We also associate to such restrictions a 1-parameter family of symmetric extensions, namely completely positive linear maps $\ca_n\to\ca_\infty$ which are right inverses of the $\r_n$, and commute with the symmetries of the triangle (see paragraph 2.2.4). By using these extensions one can prove that $\ca_\infty$ is nuclear.
Among these symmetric extensions, particularly useful are the harmonic extensions, and the affine extensions, see below for details.

\noindent We then study the ideals of $\ca_\infty$ which lead us to a classification of the irreducible representations: such representations are either finite dimensional, and in this case factorise through the maps $\r_n$, or faithful, and in this case they are restrictions of representations $\pi:\UHF(3^\infty)\to\cb(H)$ with $\pi(\ca)''=\pi(\UHF(3^\infty))''$.
Primitive ideals and tracial states are fully described too (cf. sections \ref{primid} and \ref{traces}).

\noindent As mentioned above, our description of $\ca_\infty$ rests upon purely topological data. However, the self-similar structure provides a harmonic structure.
Indeed the symmetric energy form on $\ca_0=\bc(V_0)$ generates, by self-similarity, a sequence of energy forms on $\ca_n$, while the harmonic extension gives the compatibility between the energies of different levels, finally providing a densely defined closed symmetric Dirichlet form on $\ca_\infty$.
In other terms, topology plus self-similarity endow $\ca_\infty$ with a further structure  given by the Dirichlet form and a dense sub-algebra, which consists of the elements with finite energy.

\noindent The last step consists in a further enrichment of the structure of $\ca_\infty$, which has a metric content. It is based on the observation that the data $(\ch,D)$ of the spectral triple on the classical gasket considered in \cite{GuIs16} works perfectly well for the noncommutative gasket: $\ca_\infty$ acts naturally on $\ch$ and the bounded commutator property is satisfied for the dense $*-$algebra consisting of the affine extension of the elements of $\ca_n$, $n\in\bn$.
As mentioned above, the pair $(\ch,D)$ adds a metric information, since we prescribed that all edges of the classical gasket are eigenvectors of $|D|$ with eigenvalues equal to the inverse of the length of the edge in the Euclidean metric of $\mathbb{R}^2$.

\noindent We obtain that $(\ca_\infty,L)$, with $L(a)=\|[D,a]\|$, is a quantum metric space in the sense of Rieffel \cite{Rief}.
Besides of some results which are directly inherited form the commutative case, such as the value of the metric dimension being $\log3/\log2$, we show that the spectral triple fits well with the features of $\ca_\infty$ described above. The trace on $\ca_\infty$ obtained from the Dixmier trace via the Connes trace formula is simply the restriction of the trace on $\UHF(3^\infty)$. Finally, with a formula already used in \cite{CGIS1,GuIs16}, one also recovers from $D$ the self-similar Dirichlet form, explicitly described above.

\noindent We conclude this introduction mentioning that some preliminary forms of the results contained in this paper have been illustrated in 2017 during conferences in Toulouse, Cornell and Warsaw.

\section{The algebras of the classical and quantum Sierpinski gasket}

\subsection{A functional description of the classical Sierpinski gasket}

The Sierpinski gasket may be defined via three similitudes $w_j$, $j=1,2,3$, of the Euclidean plane with contraction parameter equal to $1/2$ and centered in the vertices of an equilateral triangle. Such maps give a contraction $W$ on the space of compact subsets of $\br^2$ endowed with the Hausdorff distance, $W(C)=\cup_{i=1,2,3}w_i(C)$, and the gasket $K$ is defined as the unique fixed point. Setting $V_0$ as the set of vertices of the triangle, the vertices $V_n$ are defined inductively via $V_{n+1}=W(V_n)$. Since $V_0$ consists of fixed points, we easily get $V_n\subset V_{n+1}$. We shall also consider the set  $E_1=\{\overrightarrow{e_1},\overrightarrow{e_2},\overrightarrow{e_3},\overleftarrow{e_1},\overleftarrow{e_2},\overleftarrow{e_3}\}$ consisting of the six oriented edges of the triangle, and the families $E_n$ of oriented edges of level $n$ defined inductively by $E_{n+1}=W(E_n)$. We note in passing that with our notation any geometrical edge is counted twice in $E_n$. For each oriented edge $e$ we denote by $e^+$, resp. $e^-$ the target, resp. the source of $e$.

The self-similarity equation $K=\bigcup_{j=1,2,3}w_j(K)$ gives a family of relations for the set of vertices (essential fixed points), namely $V_{n+1}=\bigcup_{j=1,2,3}w_j(V_n)$.
The non triviality of the construction lies in the non disjointness of the  three copies of $V_n$, namely we get a non-injective projection $\psi_{n+1}:w_1(V_n)\sqcup w_2(V_n)\sqcup w_3(V_n)\to V_{n+1}$.
From the functional point of view this gives rise to a non-surjective embedding $\psi_{n+1}^*:\bc(V_{n+1})\hookrightarrow\bc^3\otimes \bc(V_{n})$.
Let us now consider the map $\psi_{1}:w_1(V_0)\sqcup w_2(V_0)\sqcup w_3(V_0)\to V_{1}$  described pictorially in fig. \ref{psi1}.

\begin{figure}
\begin{tikzpicture}
\draw[gray, thick] (0,0) -- (2,0);
\draw[gray, thick] (0,0) -- (1,1.73);
\draw[gray, thick] (2,0) -- (1,1.73);
\filldraw[black] (0,0) circle (2pt) node[anchor=east] {$2$};
\filldraw[black] (2,0) circle (2pt) node[anchor=north] {$1$};
\filldraw[black] (1,1.73) circle (2pt) node[anchor=east] {$3$};
\filldraw[black] (1,0.86) circle (0pt) node[anchor=center] {$1$}; 
\draw[gray, thick] (1.5,2.6) -- (3.5,2.6);
\draw[gray, thick] (2.5,4.33) -- (1.5,2.6);
\draw[gray, thick] (2.5,4.33) -- (3.5,2.6);
\filldraw[black] (2.5,4.33) circle (2pt) node[anchor=south] {$2$};
\filldraw[black] (3.5,2.6) circle (2pt) node[anchor=west] {$3$};
\filldraw[black] (1.5,2.6) circle (2pt) node[anchor=east] {$1$};
\filldraw[black] (2.5,3.46) circle (0pt) node[anchor=center] {$2$}; 
\draw[gray, thick] (3,0) -- (5,0);
\draw[gray, thick] (3,0) -- (4,1.73);
\draw[gray, thick] (4,1.73) -- (5,0);
\filldraw[black] (3,0) circle (2pt) node[anchor=north] {$3$};
\filldraw[black] (5,0) circle (2pt) node[anchor=west] {$2$};
\filldraw[black] (4,1.73) circle (2pt) node[anchor=west] {$1$};
\filldraw[black] (4,0.86) circle (0pt) node[anchor=center] {$3$}; 
\draw[black, very thick,->] (5.5,2.16) -- (6.5,2.16);
\filldraw[black] (6,2.16) circle (0pt) node[anchor=south] {$\psi_1$};
%
\draw[gray, thick] (7,0) -- (9,0);
\draw[gray, thick] (9,0) -- (11,0);
\draw[gray, thick] (7,0) -- (8,1.73);
\draw[gray, thick] (8,1.73) -- (10,1.73);
\draw[gray, thick] (8,1.73) -- (9,3.46);
\draw[gray, thick] (9,0) -- (8,1.73);
\draw[gray, thick] (10,1.73) -- (9,3.46);
\draw[gray, thick] (9,0) -- (10,1.73);
\draw[gray, thick] (10,1.73) -- (11,0);
\filldraw[black] (7,0) circle (2pt) node[anchor=east] {12};
\filldraw[black] (9,3.46) circle (2pt) node[anchor=south] {22};
\filldraw[black] (11,0) circle (2pt) node[anchor=west] {32};
\filldraw[black] (10,1.73) circle (2pt) node[anchor=west] {31=23};
\filldraw[black] (9,0) circle (2pt) node[anchor=north] {11=33};
\filldraw[black] (8,1.73) circle (2pt) node[anchor=east] {21=13};
\end{tikzpicture}
\caption{The map $\psi_1$}
 \label{psi1}
\end{figure}
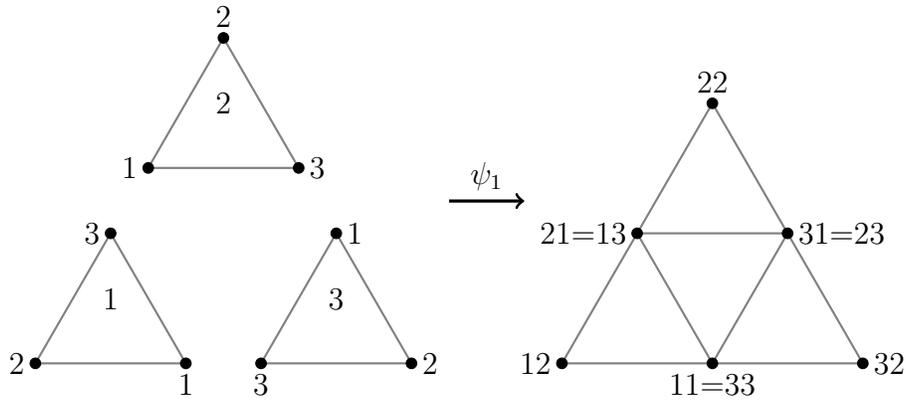

\noindent
From the functional point of view, $\psi_1$ gives rise to an embedding $\psi_{1}^*:\bc(V_{1})\hookrightarrow\bc^3\otimes \bc(V_{0})$. By identifying $\bc^3$ with the diagonal sub-algebra of $ M_3(\bc)$, we get the embedding $\bc(V_{1})\hookrightarrow M_3(\bc)\otimes\bc(V_{0})$, which we denote by the same symbol. The Bratteli diagram for the embedding $\psi_1^*$ is in Fig. \ref{diag01},
     \begin{figure}[ht]
 	 \centering
\begin{tikzpicture}
\draw[gray, thick] (-0.86,0.5) -- (0.86,0.5);
\draw[gray, thick] (0,-1) -- (-0.86,0.5);
\draw[gray, thick] (0,-1) -- (0.86,0.5);
\draw[gray, thick] (0,0.5) -- (0,2);
\draw[gray, thick] (-1.73,-1) -- (-0.43,-0.25);
\draw[gray, thick] (1.73,-1) -- (0.43,-0.25);
\filldraw[black] (0,0.5) circle (2pt) ;
\filldraw[black] (0.43,-0.25) circle (2pt) ;
\filldraw[black] (-0.43,-0.25) circle (2pt) ;
\draw[black] (0,2) circle (2pt) ;
\draw[black] (1.73,-1) circle (2pt) ;
\draw[black] (-1.73,-1) circle (2pt) ;
\draw[black] (0,-1) circle (2pt) ;
\draw[black] (0.86,0.5) circle (2pt) ;
\draw[black] (-0.86,0.5) circle (2pt) ;
\end{tikzpicture}
	  \caption{Bratteli diagram of the inclusion $\psi_1^*:\bc(V_{1})\hookrightarrow M_3(\bc)\otimes\bc(V_{0})$}
	 \label{diag01}
     \end{figure}
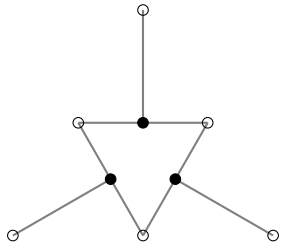
where black dots have label 3 and describe the minimal central projections in $ M_3(\bc)\otimes\bc(V_0)$, while white dots have label 1 and describe the minimal (central) projections in $\bc(V_1)$.

The non-injectivity of the map $\psi_n$ is  related to a suitable gluing of the 3 external vertices of each copy of $V_n$ producing the 6 vertices of level 1 in $V_{n+1}$. Apart from that, the map is indeed injective.
This implies that the algebras $\bc(V_n)$ satisfy the inductive relations
\begin{equation}\label{identification}
\bc(V_{n+1}\setminus V_1)\cong\bc^3\otimes \bc(V_{n}\setminus V_0),\quad n\geq1,
\end{equation}
and the map $\psi^*_{n+1}$ acts diagonally w.r.t. the decomposition $\bc(V_{n+1})=\bc(V_1)\oplus\bc(V_{n+1}\setminus V_1)$;
it acts as $\psi_1^*$ on the first summand  and consists of the identification \eqref{identification} on the second summand:
$$
\psi_{n+1}^*(\bc(V_{n+1}))=
\psi^*_1(\bc(V_1))\oplus \bc^3\otimes\bc(V_n\setminus V_0)
\hookrightarrow \bc^3\otimes\bc(V_n).
$$

We have just shown that the map $\psi_1^*$ determines all the embeddings $\psi_n^*$. We now show how the inductive relations \eqref{identification} together with the map $\psi_1^*$ completely determine the algebra of continuous functions on $K$.

Setting $\psi^*_{1,0}:=\psi^*_1$ and $\psi^*_{n+1,0}:=(id\otimes \psi^*_{n,0})\circ \psi^*_{n+1}:\bc(V_{n+1})\to(\bc^3)^{\otimes (n+2)}$, defines inductively the maps $\psi^*_{n,0}:\bc(V_{n})\to(\bc^3)^{\otimes (n+1)}$. Composing again with the natural embedding of $(\bc^3)^{\otimes (n+1)}$ into the inductive limit $C^*$-algebra $(\bc^3)^{\otimes \infty}$, we get  injective maps $\f_n$ from $\bc(V_n)$ into $(\bc^3)^{\otimes \infty}$, namely all $\bc(V_n)$'s may be identified with suitable unital sub-algebras of $(\bc^3)^{\otimes \infty}$.

We may also identify the continuous functions on $\Sigma_\infty=$ the infinite words on three letters $\cong$ the boundary of the infinite ternary tree, with $(\bc^3)^{\otimes\infty}$. Since the classical analysis of Kigami \cite{Kiga} provides a continuous projection $P:\Sigma_\infty\to K$, by which $C(K)$ embeds in $(\bc^3)^{\otimes\infty}$, from now on we shall always consider the algebras $\bc(V_n)$ and $C(K)$ as subalgebras of $(\bc^3)^{\otimes\infty}$. Our aim now is to show that in a natural sense $\bc(V_n)$ converges to $C(K)$
as $n\to\infty$.


\begin{thm}\label{Kinfinity}
$C(K) = \{a\in(\bc^3)^{\otimes \infty}:a=\lim\f_n(a_n), a_n\in\bc(V_n)\}$.
\end{thm}
\begin{proof}
If $f\in\bc(V_n)$,  $\f_n(f)$ is a cylindrical function on $\Sigma_\infty$, namely its value on a point $\sigma\in\Sigma_\infty$ is completely determined by the first $n+1$ letters of $\sigma$. Therefore $f$ may be identified with a function on $K$ which is constant on the open cells $C$ of level $n+1$, with $f(C)=f(v)$ if $v\in\partial C\cap V_n$, and is double-valued on $V_{n+1}\setminus V_n$. Therefore a sequence $\{a_n\}$, $a_n\in\bc(V_n)$, such that $\f_n(a_n)$ converges in $(\bc^3)^{\otimes\infty}$ gives rise to a continuous function on $\Sigma_\infty$ which respects all identifications, namely to a continuous function on $K$. Conversely, given a function $f\in C(K)$, one may set $a_n$ to be its restriction to $V_n$. By the uniform continuity of $f$, $\f(a_n)$ converges in norm to a uniformly continuous function on $V$ which coincides with $f$ on all vertices.
\end{proof}

\begin{rem}
As observed in the proof above,  for any closed cell $C$ of level $n\in\bn$ there exists a unique vertex $v$ of level $(n-1)$ such that $v\in C$. Therefore, $\bc(V_n)$  can be identified with the algebra of functions on the gasket which are constant on open cells of level $n+1$ and have jump discontinuities at most on $V_{n+1}\setminus V_{n}$.
This shows that, as sub-algebras of $(\bc^3)^{\otimes\infty}$, $\bc(V_n)\cap\bc(V_m)=\bc$ if $n\ne m$, and
$\bc(V_n)\cap C(K)=\bc$ for any $n\in\bn$.
\end{rem}

\subsubsection{Numbering  the vertices}

We now give a more concrete description of the algebras $\bc(V_n)$ as sub-algebras of $(\bc^3)^{\otimes(n+1)}$. Let us describe $\psi_1$
as in Fig\,\ref{psi1}.
Correspondingly, setting $\a^0_j=e_{j},j=1,2,3$, the characteristic functions in $\bc(V_0)=\bc^3$  of the vertices in $V_0$, denoting by $\a^1_j,j=1,2,3$, the characteristic functions in $\bc(V_1)$  of the vertices in $V_0$ and by  $\b^1_j,j=1,2,3$, the characteristic functions in $\bc(V_1)$ of the vertices in $V_1\setminus V_0$, we get $\psi^*_1(\a^1_j)=e_j\otimes\a^0_2$ and $\psi^*_1(\b^1_j)=e_{j}\otimes \a^0_{1}+e_{(j+2)} \otimes \a^0_{3}$, where the elements on the right are in $\bc^3\otimes\bc^3$, the $e_{j}$'s denote the matrix units in $\bc^3$, and the indices are meant mod $3$.  Let us now denote by $\a^n_j,j=1,2,3$, the characteristic functions in $\bc(V_n)$  of the vertices in $V_0$ and by  $\b^n_j,j=1,2,3$, the characteristic functions in $\bc(V_n)$ of the vertices in $V_1\setminus V_0$, $n\in\bn$. Then, by the inductive definition of $\psi^*_{n+1}$, the algebra $\psi_{n+1}^*(\bc(V_{n+1}))\subset\bc^3\otimes\bc(V_n)$ is generated by the elements
\begin{align*}
\{\psi_{n+1}^*(\a^{n+1}_j)=e_{j}\otimes\a^n_2,\
&\psi_{n+1}^*(\b^{n+1}_j)=e_{j}\otimes\a^n_1+e_{j+2} \otimes \a^n_3,
\ e_j\otimes \xi , \\
&j=1,2,3, \xi\in\bc(V_n\setminus V_0)\}.
\end{align*}
Now, for $\s\in\Sigma_n := \{1,2,3\}^n$, we use the shortcut $\s:=e_{\s_1}\otimes e_{\s_2}\otimes\dots\otimes e_{\s_n}$, therefore $\psi_1^*(\bc(V_1))\subset\bc^3\otimes\bc^3$ is generated by the elements $\{j2,\ j1+(j+2)3:j=1,2,3\}$.

\begin{lem}\label{indices}
The algebras $\psi^*_{n,0}(\bc(V_n))\subset(\bc^3)^{\otimes(n+1)}$ are linearly generated as follows
$$
\psi^*_{n,0}(\bc(V_n))=
\langle \psi^*_{n,0}(\a^n_j),\
x\otimes \psi^*_{k,0}(\b^{k}_j):
j=1,2,3, k=1,\dots n,
x\in(\bc^3)^{\otimes (n-k)}\rangle,
$$
with $\psi^*_{n,0}(\a^n_j)=j2^n$, $\psi^*_{k,0}(\b^k_j)=(j1+(j+2)3)2^{k-1}
=\b^1_j\otimes 2^{k-1}$, $j=1,2,3$.
\end{lem}
\begin{proof}
We  prove the formula by induction.
The statement is true for $n=1$. Assuming it is true for $n$, one gets
\begin{align*}
\psi^*_{n+1,0}(\a^{n+1}_j)
&=(id\otimes \psi^*_{n,0})\circ \psi^*_{n+1}(\a^{n+1}_j)
=(id\otimes \psi^*_{n,0})(e_j\otimes \a^{n}_2)
=j2^{n+1},\\
\psi^*_{n+1,0}(\b^{n+1}_j)
&=(id \otimes \psi^*_{n,0})\circ \psi^*_{n+1}(\b^{n+1}_j)
=(id \otimes \psi^*_{n,0})(e_{j}\otimes\a^n_1+e_{j+2} \otimes \a^n_3)\\
&=e_{j}\otimes\psi^*_{n,0}(\a^n_1)+e_{j+2} \otimes \psi^*_{n,0}(\a^n_3)
=j12^n+(j+2)32^n=\b^1_j\otimes 2^{n}.
\end{align*}
The formula above states $\psi^*_{n,0}(\bc(V_n\setminus V_1))= \langle x\otimes (j1+(j+2)3)2^{k-1},j=1,2,3, k=1,\dots n-1, x\in(\bc^3)^{\otimes (n-k)}\rangle$. Then
\begin{align*}
\psi^*_{(n+1),0}(\bc(&V_{n+1}\setminus V_1))
=(id\otimes \psi^*_{n,0})\cdot \psi^*_{n+1}(\bc(V_{n+1}\setminus V_1))
=(id\otimes \psi^*_{n,0})(\bc^3\otimes\bc(V_{n}\setminus V_1))\\
&=\langle y\otimes x\otimes (j1+(j+2)3)2^k,j=1,2,3,\ k=0,\dots n-2, y\in\bc^3, x\in(\bc^3)^{\otimes (n-1-k)}\rangle\\
&=\langle x\otimes (j1+(j+2)3)2^k:j=1,2,3,\ k=0,\dots n-2, x\in(\bc^3)^{\otimes (n-k)}\rangle
\end{align*}
\end{proof}

The description above provides a labelling of the vertices of level $n$ by identifying them with their characteristic functions hence with the labels
\begin{equation}\label{CVn}
V_n=\{j2^n,\ \s (j1+(j+2)3)2^{k-1}:j=1,2,3, k=1,\dots n, \s\in \Sigma_{n-k} \}
\end{equation}

\begin{rem}\label{GenVertex}
In the previous labelling of vertices, the exponent $p$ of 2 on the right side of the label of a given vertex $v$ describes its age in $V_n$, i.e. $p=0$ means that $v$ firstly appeared in $V_n$, $p=1$ that it firstly appeared in $V_{n-1}$, and so on. In this way any vertex in $V_{n-1}$ appears also in $V_n$ simply by adding a 2 on the right to its label. This gives a non-unital embedding of $\bc(V_n)$ in $\bc(V_{n+1})$. Dually, we get a restriction map for functions, which is a unital $*$-algebra morphism from $\bc(V_{n+1})$ to $\bc(V_n)$, in particular the restriction of the characteristic function of a vertex is 0 if $p=0$ and is the vertex with the same label, but $p$ replaced by $p-1$, if $p>0$. Restrictions may be composed, and the restriction of the characteristic function of a vertex $v\in V_{n+k}$ to $\bc(V_n)$ is non-zero if and only if $v$ is at least $k$-generations old. In this case its restriction is obtained simply by subtracting $k$ to the exponent of 2.
\end{rem}

%
%
%
%

\subsection{The quantum gasket}

\subsubsection{A result on sequences of subalgebras of a given C$^*$-algebra}

Suppose we have a C$^*$-algebra $\cb$ and a sequence of sub-C$^*$-algebras $\cb_n$. We may define
\begin{align}
\cb_\infty&=\{b\in\cb:b=\lim_n b_n, b_{n}\in\cb_{n}\},\label{Ainfty-}
\\
\cb^\infty&=\bigcap_{n\in\bn}C^*(\bigcup_{k\geq n}\cb_k)\label{Ainfty+}
\end{align}

\begin{prop}\label{Ainfty}
Both $\cb_\infty$ and $\cb^\infty$ are C$^*$-algebras, and $\cb_\infty\subseteq\cb^\infty$.
\end{prop}
\begin{proof}
The C$^*$-algebra property of $\cb^\infty$ is obvious by construction.
The $*-$algebra property of $\cb_\infty$ follows by the properties of norm-limits. As for its closure, let us observe that $b=\lim_nb_n$, $b_n\in\cb_n$ {\it iff} $d(b,\cb_n)\to0$, which in turn is equivalent to $\forall\eps>0\,\exists n_\eps:n>n_\eps\Rightarrow d(b,\cb_n)<\eps$. Therefore, given $b\in\overline{\cb_\infty}$, $\eps>0$, we find $b_\eps\in\cb_\infty: d(b,b_\eps)<\eps/2$ and $n_\eps\in\bn$:  $n>n_\eps\Rightarrow d(b_\eps,\cb_n)<\eps/2$. As a consequence $d(b,\cb_n)<\eps$ for $n>n_\eps$, namely $b\in\cb_\infty$.
\end{proof}
\begin{rem}\label{Xinfinity}
Let us observe that $\cb_\infty$ can be strictly smaller than $\cb^\infty$.
For example, setting $\cb_n=\bc I$ for $n$ odd and $\cb_n=\cb$ for $n$ even, $\cb_\infty=\bc I$ while $\cb^\infty=\cb$.
\end{rem}

\subsubsection{The algebras $\ca_n$}

We have shown that the family of algebras $\bc(V_n)$, their embeddings in $(\bc^3)^{\otimes\infty}$, and finally the $C^*$-algebra $C(K)$ can be defined in terms of the map $\psi^*_1$ and of the inductive relations \eqref{identification}. We propose now a simple quantisation of this procedure, and call the corresponding $C^*$-algebra $\ca_\infty$ (the algebra of functions on) the quantum gasket. In the definition below $\ca_n$ is the quantised version of $\bc(V_n)$ and $\car_n$ is the quantised version of $\bc(V_n\setminus V_1)$.

\begin{dfn}[The algebras $\ca_n$]\label{IndDef}
Set $\car_1=\{0\}$ and define inductively $\car_n$ such that $\car_{n+1} = M_3(\bc) \otimes (\bc^3\oplus \car_n)$, $n\in\bn$. Then set $\ca_0=\bc^3$, $\ca_n=\bc^3\oplus\bc^3\oplus\car_n$, $n\in\bn$.
\end{dfn}

\begin{lem}\label{AntoUHF}
The map $\psi^*_1$ for the classical gasket determines inclusions $\psi^*_{n+1}:\ca_{n+1}\hookrightarrow M_3(\bc)\otimes\ca_n$. Therefore we get embeddings $\ca_n
\mathop{\longrightarrow}\limits^{\psi^*_{n,0}}
M_3(\bc)^{\otimes(n+1)} \mathop{\longrightarrow}\limits^{i_{n+1}}\UHF(3^\infty)$, where
$\psi^*_{n,0}$ is defined inductively by  $\psi^*_{1,0}:=\psi^*_1$, $\psi^*_{n+1,0}:=(id\otimes \psi^*_{n,0})\circ \psi^*_{n+1}$  and $i_n$ is the natural embedding of
$M_3(\bc)^{\otimes n}$ into the inductive limit $\UHF(3^\infty)$.
\end{lem}
\begin{proof}
According to the definition above, $\psi^*_{1}$ is a map from $\bc^3\oplus\bc^3$ into $\bc^3\otimes\bc^3$. Then if $(x,y,0)\in\ca_{n+1}=\bc^3\oplus\bc^3\oplus\car_{n+1}$, $n\in\bn$, we set $\psi^*_{n+1}(x,y,0)=\psi^*_1(x,y)\oplus0\oplus0\in M_3(\bc) \otimes \big(\bc^3\oplus\{0\}\oplus\{0\}\big)\subset M_3(\bc)\otimes\ca_n$, and define $\psi^*_{n+1}|_{\car_{n+1}}$  simply as the identification of $\car_{n+1}=M_3(\bc)\otimes(\bc^3\oplus \car_n)$ with $M_3(\bc)\otimes(\{0\}\oplus\bc^3\oplus \car_n)\subset M_3(\bc)\otimes\ca_n$.
The ranges of the maps $\psi^*_{n,0}$ are obvious by definition.
\end{proof}

In the following we identify $\ca_n$ with its image in $M_3(\bc)^{\otimes(n+1)}\subset\UHF(3^\infty)$.

\begin{rem}
As mentioned above, $\car_n$ is the quantised version of $\bc(V_n\setminus V_1)$, in fact $\bc(V_{n+1}\setminus V_1) = \bc^3 \otimes (\bc^3\oplus \bc(V_n\setminus V_1))$, $n\in\bn$, namely passing from $\bc(V_n\setminus V_1)$ to $\car_n$ consists in replacing the leftmost factor $\bc^3$ with $M_3(\bc)$. Then the $\bc(V_n)$ can be defined in analogy with the definition of the $\ca_n$'s: $\bc(V_n)=\bc^3\oplus\bc^3\oplus\bc(V_{n+1}\setminus V_1)$, $n\in\bn$. This shows in particular that $\bc(V_n)\subset \ca_n\subset M_3(\bc)^{\otimes n}\otimes \bc(V_0)$.
\end{rem}
\begin{thm}\label{AinfAlgebra}
The set $\displaystyle
\ca_\infty=\{a\in\UHF(3^{\infty}):a=\lim_n a_n, a_{n}\in\ca_{n}\}$
is a $C^*$-algebra containing $C(K)$ as an abelian sub-algebra.
\end{thm}

\begin{proof}
The C$^*$-algebra property follows by Proposition \ref{Ainfty}.
Finally, by construction $\bc(V_n)\subset\ca_n$, hence the last statement follows by Theorem \ref{Kinfinity}.
\end{proof}
According to the Remark above, one easily sees that $\ca_0$, resp. $\ca_1$, can be identified with $\bc(V_0)$, resp. $\bc(V_1)$, and $\ca_2$  is the first noncommutative algebra among the $\ca_n$'s.
Fig. \ref{fig:diagrams} compares the Bratteli dagrams for the inclusions
$\bc(V_2)\subset M_3(\bc)\otimes\bc(V_1)$ of the classical gasket, and
$\ca_2\subset M_3(\bc)\otimes\ca_1$ of the quantum gasket.
Black dots have label 3 and describe the minimal central projections in $ M_3(\bc)\otimes\bc(V_1)$ resp. $ M_3(\bc)\otimes\ca_1$, white dots  describe the minimal (central) projections in $\bc(V_2)$ resp. $\ca_2$.

 \begin{figure}[ht]
\begin{tikzpicture}
\draw[gray, thick] (-1.73,1) -- (1.73,1);
\draw[gray, thick] (0,-2) -- (-1.73,1);
\draw[gray, thick] (0,-2) -- (1.73,1);
\draw[gray, thick] (0,1) -- (0,4);
\draw[gray, thick] (-3.46,-2) -- (-0.86,-0.5);
\draw[gray, thick] (3.46,-2) -- (0.86,-0.5);
\draw[gray, thick] (0,-0.5) -- (0,-1);
\draw[gray, thick] (0.43,0.25) -- (0.86,0.5);
\draw[gray, thick] (-0.43,0.25) -- (-0.86,0.5);
\draw[gray, thick] (0.43,0.25) -- (0.68,0.68);
\draw[gray, thick] (0.43,0.25) -- (0.93,0.25);
\draw[gray, thick] (-0.43,0.25) -- (-0.68,0.68);
\draw[gray, thick] (-0.43,0.25) -- (-0.93,0.25);
\draw[gray, thick] (0,-0.5) -- (-0.25,-0.93);
\draw[gray, thick] (0,-0.5) -- (0.25,-0.93);
\filldraw[black] (0,-0.5) circle (2pt) ;
\filldraw[black] (0.43,0.25) circle (2pt) ;
\filldraw[black] (-0.43,0.25) circle (2pt) ;
\filldraw[black] (0,1) circle (2pt) ;
\filldraw[black] (0.86,-0.5) circle (2pt) ;
\filldraw[black] (-0.86,-0.5) circle (2pt) ;
\draw[black] (0,4) circle (2pt) ;
\draw[black] (3.46,-2) circle (2pt) ;
\draw[black] (-3.46,-2) circle (2pt) ;
\draw[black] (0,-2) circle (2pt) ;
\draw[black] (1.73,1) circle (2pt) ;
\draw[black] (-1.73,1) circle (2pt) ;
\draw[black] (0,-1) circle (2pt) ;
\draw[black] (0.86,0.5) circle (2pt) ;
\draw[black] (-0.86,0.5) circle (2pt) ;
\draw[black] (0.68,0.68) circle (2pt) ;
\draw[black] (0.93,0.25) circle (2pt) ;
\draw[black] (-0.68,0.68) circle (2pt) ;
\draw[black] (-0.93,0.25) circle (2pt) ;
\draw[black] (-0.25,-0.93) circle (2pt) ;
\draw[black] (0.25,-0.93) circle (2pt) ;
\draw[gray, thick] (6.27,1) -- (9.73,1);
\draw[gray, thick] (8,-2) -- (6.27,1);
\draw[gray, thick] (8,-2) -- (9.73,1);
\draw[gray, thick] (8,1) -- (8,4);
\draw[gray, thick] (4.54,-2) -- (7.14,-0.5);
\draw[gray, thick] (11.46,-2) -- (8.86,-0.5);
\draw[gray, thick] (8,-0.5) -- (8,-1);
\draw[gray, thick] (8.43,0.25) -- (8.86,0.5);
\draw[gray, thick] (7.57,0.25) -- (7.14,0.5);
\filldraw[black] (8,-0.5) circle (2pt) ;
\filldraw[black] (8.43,0.25) circle (2pt) ;
\filldraw[black] (7.57,0.25) circle (2pt) ;
\filldraw[black] (8,1) circle (2pt) ;
\filldraw[black] (8.86,-0.5) circle (2pt) ;
\filldraw[black] (7.14,-0.5) circle (2pt) ;
\draw[black] (8,4) circle (2pt) ;
\draw[black] (11.46,-2) circle (2pt) ;
\draw[black] (4.54,-2) circle (2pt) ;
\draw[black] (8,-2) circle (2pt) ;
\draw[black] (9.73,1) circle (2pt) ;
\draw[black] (6.27,1) circle (2pt) ;
\draw[black] (8,-1) circle (2pt) ;
\draw[black] (8.86,0.5) circle (2pt) ;
\draw[black] (7.14,0.5) circle (2pt) ;
\end{tikzpicture}
   \caption{\quad
   $\bc(V_2)\subset M_3(\bc)\otimes\bc(V_1)$ \qquad\qquad\qquad\qquad
   $\ca_2\subset M_3(\bc)\otimes\ca_1$ \hskip 2cm }
   \label{fig:diagrams}
 \end{figure}

By the observation above we get $\ca_0=\langle\a^0_j:j=1,2,3\rangle$ and $\ca_1=\langle\a^1_j,\b^1_j:j=1,2,3\rangle$, cf. Lemma \ref{indices}.
Moreover, the following Lemma holds.

\begin{lem}\label{indicesBis}
The algebras $C(V_n)\subset\ca_n\subset M_3(\bc)^{\otimes(n+1)}$ are linearly generated as follows
\begin{align*}
C(V_n) &= \langle \a^n_j,\ x\otimes \b^{k}_j:j=1,2,3, k=1,\dots n, x\in (\bc^3)^{\otimes (n-k)} \rangle.
\\
\ca_n &= \langle \a^n_j,\ x\otimes \b^{k}_j:j=1,2,3, k=1,\dots n, x\in M_3(\bc)^{\otimes (n-k)} \rangle.
\end{align*}
where $\a^n_j=j2^n, \b^k_j=(j1+(j+2)3)2^{k-1}\in(\bc^3)^{\otimes(k+1)}\subset M_3(\bc)^{\otimes(k+1)}$.
\end{lem}
\begin{proof}
The statement about $C(V_n)$ directly follows by equation \eqref{CVn}.
The statement about $\ca_n$ amounts to say that $\ca_n=\bc^3\oplus\bc^3\oplus\car_n$, with $\car_n = \langle
x \otimes \b^{k}_j:j=1,2,3, k=1,\dots n-1, x\in M_3(\bc)^{\otimes (n-k)} \rangle$.
The last equation is clearly true for $n=1$, and, assuming it for $n$, we get
\begin{align*}
& \langle
 x\otimes \b^{k}_j:j=1,2,3, k=1,\dots n, x\in M_3(\bc)^{\otimes (n+1-k)} \rangle \\
 =&\langle x\otimes \b^{n}_j:j=1,2,3, x\in M_3(\bc)\rangle
\oplus M_3(\bc)\otimes \car_n\\
=&  M_3(\bc) \otimes \big( \bc^3\oplus \car_n\big)=\car_{n+1}.
\end{align*}
 The thesis follows by induction.
\end{proof}

\begin{cor}\label{MASA1}
The algebra $C(V_n)$ is maximal abelian in $\ca_n$.
\end{cor}
\begin{proof}
Directly follows by Lemma \ref{indicesBis} and the following equalities:  
$\a^n_j\a^{n}_{j'}=\d_{j,j'}\,\a^n_j$,  $\a^n_j \cdot x\otimes \b^k_i=0$ for any  $i,j=1,2,3, k=1,\dots n, x\in (\bc^3)^{\otimes (n-k)}$, and $x\otimes \b^{k}_i\cdot x'\otimes \b^{k'}_{i'}=\d_{k,k'}\d_{i,i'}\,x\cdot x'\otimes \b^{k}_i$.
%
\end{proof}

We now define the morphisms $\xi_j$, $j=1,2,3$, and $\eta_{k,j}$, $j=1,2,3, k=0,\dots n-1$, on $\ca_n$ through their action on generators:
\begin{equation}\label{xi_and_eta}
\begin{split}
&
\begin{matrix}
\xi_i     &:&\a^n_j        &\longrightarrow&\d_{i,j}\\
          & &x\otimes\b^k_j&\longrightarrow&0
\end{matrix}
\qquad\Big\}\in\bc,
\qquad i,j=1,2,3, k=1,\dots n;
\\
&
\begin{matrix}
\eta_{k,i}&:&\a^n_j        &\longrightarrow&0,\\
          & &x\otimes\b^{n-m}_j&\longrightarrow&\d_{i,j}\d_{k,m}\,x
\end{matrix}
\qquad\Big\}\in M_3(\bc)^{\otimes k}
\qquad i,j=1,2,3, k,m=0,\dots n-1.
\end{split}
\end{equation}

\begin{cor}\label{multiplicativity} \-

\item[$(1)$] Any element of $\ca_n$ can be written in a unique way as
\begin{equation}\label{decomposition}
b=\sum_{j=1}^3 \xi_{j}(b)\a_j^n
+\sum_{k=0}^{n-1}\sum_{j=1}^3 \eta_{k,j}(b)\otimes \b_j^{n-k} \,,
\quad \xi_{j}(b)\in \bc, \;\eta_{k,j}(b)\in M_3(\bc)^{\otimes k}.
\end{equation}

\item[$(2)$] $\xi_j$  is a character of $\ca_n$, for any $j=1\cdots 3$, while each $\eta_{k,j}$ is a representation of $\ca_n$ onto $M_3(\bc)^{\otimes k}$, for any $j=1\cdots 3$, $k=0\cdots n-1$.\\

\item[$(3)$] The spectrum $\widehat{\ca_n}$ of $\ca_n$ has $3n$ elements,  representatives of which are the $\xi_j$ and the $\eta_{k,j}$.\\

\item[$(4)$] The set of primitive ideals of $\ca_n$ has $3n$ elements, which are the respective kernels of the $3n$ irreducible representations above.
\end{cor}

\subsubsection{Restriction maps}
For the classical gasket, the natural embedding $V_n\hookrightarrow K$ gives rise to a restriction map from $C(K)\to\bc(V_n)$ which is clearly a C$^*$-algebra morphism. Here we show that such morphism is defined also for the quantum gasket, giving rise to a family of representations of $\ca_\infty$.
We start introducing  the restriction maps from $\ca_{n+1}$ to $\ca_n$, $n\geq0$. As in the abelian case, they will consist in erasing a 2 on the right of the generators.


\begin{lem} \label{restrizione}
The completely positive contraction $ id_n\otimes e_{22}^*: M_3(\bc)^{\otimes(n+1)}\hookrightarrow  M_3(\bc)^{\otimes n} $, where $e_{22}^*$ is the state $m\in M_3(\bc)\mapsto m_{22}\in \bc$, and $id_n$ is the identity map of $M_3(\bc)^{\otimes n}$, induces a $*$-algebra epimorphism from $\ca_{n}$ onto $\ca_{n-1}$. The map $\otimes e_{22}:a\in M_3(\bc)^{\otimes n}\to a\otimes e_{22}\in M_3(\bc)^{\otimes (n+1)}$ induces a non-unital $*$-algebra morphism from $\ca_{n-1}$ in $\ca_{n}$. Clearly the former map is a left inverse of the latter.
\end{lem}
\begin{proof}
As for the first statement, the case $n=1$ is obvious. For $n\geq 2$, $id\otimes e_{22}^*$ maps
\begin{itemize}
\item[] $x\otimes \b_j^{k}$ onto $x\otimes \b_j^{k-1}$ if $k>1$,
\item[] $x\otimes \b_j^{1}$ to $0$,
\item[] $\a_j^{n}$ to $\a_j^{n-1}$.
\end{itemize}
Since the product of two  generators may be non-zero either for the product of the projection $\a^n_j$ with itself or for $x\otimes \b_j^{k}\cdot y\otimes \b_j^{k}=xy\otimes \b_j^{k}$, the map preserves multiplication.
The same arguments apply for the second statement.
\end{proof}

\medskip 

We now observe that any state $\omega$ on $M_3(\bc)$ induces a state $\omega^{\otimes \infty}$ on $\UHF(3^\infty)$. Hence we get a completely positive contraction $id_{n+1}\otimes \omega^{\otimes \infty}$ from $M_3(\bc)^{\otimes (n+1)}\otimes \UHF(3^\infty)$ onto $M_3(\bc)^{\otimes(n+1)}$. Finally, via the natural identification  of $M_3(\bc)^{\otimes (n+1)}\otimes \UHF(3^\infty)$ with $\UHF(3^\infty)$, we get a completely positive contraction $id_{n+1}\otimes \omega^{\otimes[n+2,\infty)}:\UHF(3^\infty) \to M_3(\bc)^{\otimes (n+1)} \subset \UHF(3^\infty)$.

\begin{lem}\label{omegan}
For any $a\in \UHF(3^\infty)$,
$a=\displaystyle \lim_{n\to \infty} id_{n+1}\otimes \omega^{\otimes[n+2,\infty)}(a)$.
\end{lem}
\begin{proof}
If $a=a_k\otimes I_{[k+1,\infty)}$, with $a_k\in M_3(\bc)^{\otimes k}$, then  $id_{n+1}\otimes \omega^{\otimes[n+2,\infty)}(a)=a$ for $n\geq k$.
Since such elements are dense in $\UHF(3^\infty)$, we get the thesis.
\end{proof}

%

\begin{prop}\label{restrizioni}
Let $n\geq 0$, $k\in\bn\cup\{\infty\}$. Then, setting $\rho_n=id_{n+1}\otimes (e^*_{22})^{\otimes[n+2,\infty)}I_{[n+2,\infty)}$,  $\rho_n:\ca_{n+k}\to\ca_n$ is a $C^*-$algebra morphism.
\end{prop}
\begin{proof} By Lemma \ref{restrizione}, $\r_{n+j}(\ca_{n+j+1})=\ca_{n+j}$. If $k$ is finite, $\rho_n=\rho_n\circ \rho_{n+1}\circ \dots \circ \rho_{n+k-1}$, hence  $\rho_n$ is a  morphism from $\ca_{n+k}$ to $\ca_{n}$. When $a\in\ca_\infty$, $a=\lim_n a_n$, $a_n\in\ca_n$, $\r_p(a)=\lim_n\r_p(a_n)\in\ca_p$, where we used the contraction property of the $\r_n$'s. The morphism property follows.
\end{proof}

%
%

\begin{cor}\label{Cstella}
The algebra $\ca_\infty$ may be equivalently defined as
$$\ca_{\infty}=\{a\in \UHF(3^\infty) |  \rho_n(a)\in\ca_n\}.$$
\end{cor}
\begin{proof}
The statement follows by Lemma \ref{omegan} and Proposition \ref{restrizioni}.
\end{proof}

\subsubsection{Symmetric extensions}\label{sec:symext}

In this section we describe a one-parameter family of completely positive maps which are right-inverses of the restriction maps described in the previous section. The word symmetric refers to the fact that they commute with the symmetry group of the triangle. We shall be particularly interested to the cases $t=3/5$ and $t=1/2$.

For any $t\in [1/2,1)$ we define the positivity preserving map
\begin{equation}\label{t-ext}
\l^{(t)}:\bc(V_0)\to\bc(V_1), \quad
\l^{(t)}(\a^0_j)=\a^1_j+(1-t) \b^1_{j}+(1-t)\b^1_{j+1}+(2t-1)\b^1_{j-1} .
\end{equation}
Let now $s$ be a symmetry of the triangle, and consider its action both on the restriction to $V_0$ and to $V_1$. Then we get the adjoint maps $s_0^*:\bc(V_0)\to\bc(V_0)$ and $s_1^*:\bc(V_1)\to\bc(V_1)$. Indeed  $s_1^*\circ\lambda^t=\lambda^t\circ s_0^*$, namely the symmetric property mentioned above.
We then  define the completely positive unital contractions
$$
\l^{(t)}_n
=id_n\otimes \l^{(t)}: M_3(\bc)^{\otimes n}\otimes \bc(V_0)\subset \UHF(3^\infty) \longrightarrow M_3(\bc)^{\otimes n}\otimes \bc(V_1)\subset \UHF(3^\infty).
$$
For any $a\in M_3(\bc)^{\otimes n}\otimes \bc(V_0)\cong\bc(V_0, M_3(\bc)^{\otimes n})$, we set
\begin{equation}\label{osc}
\osc{a}=\max_{i\ne j}\|a(v_i)-a(v_j)\|,
\end{equation}
and observe that $\osc$ is also defined on $\bc(V_n)$ since $\bc(V_n)\subset  M_3(\bc)^{\otimes n}\otimes \bc(V_0)$.
\begin{prop}\label{ext1}
The following properties hold:
\begin{enumerate}
\item $\l^{(t)}_n(I)=I$.
\item If $f\in \bc(V_n)$, $\osc(f)=\max_{e\in E_n}|f(e^+)-f(e^-)|$.
\item If $a\in M_3(\bc)^{\otimes n}\otimes \bc(V_0)$, $\osc \big(\l^{t}_n(a)\big)\leq t\osc (a)$ and $\|a-\l^{(t)}_n(a)\|\leq t\osc(a)$.
\end{enumerate}
\end{prop}
\begin{proof}
Properties (1) and (2) are obvious. 
We now prove the first inequality of property (3).
Since $\l^{(t)}_n(a)\in M_3(\bc)^{\otimes n}\otimes \bc(V_1)\cong \bc(V_1, M_3(\bc)^{\otimes n})$, we denote by  $a_i,b_i\in M_3(\bc),i=1,2,3$, the values of $\l^{(t)}_n(a)$ on $V_1$ as in fig. \ref{fig-iterosc} and exploit the relations \eqref{t-ext}:
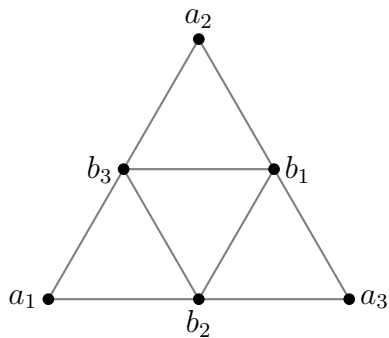
\begin{figure}[ht]
\begin{tikzpicture}
\draw[gray, thick] (0,0) -- (2,0);
\draw[gray, thick] (2,0) -- (4,0);
\draw[gray, thick] (0,0) -- (1,1.73);
\draw[gray, thick] (1,1.73) -- (3,1.73);
\draw[gray, thick] (1,1.73) -- (2,3.46);
\draw[gray, thick] (2,0) -- (1,1.73);
\draw[gray, thick] (3,1.73) -- (2,3.46);
\draw[gray, thick] (2,0) -- (3,1.73);
\draw[gray, thick] (3,1.73) -- (4,0);
\filldraw[black] (0,0) circle (2pt) node[anchor=east] {$a_1$};
\filldraw[black] (2,3.46) circle (2pt) node[anchor=south] {$a_2$};
\filldraw[black] (4,0) circle (2pt) node[anchor=west] {$a_3$};
\filldraw[black] (3,1.73) circle (2pt) node[anchor=west] {$b_1$};
\filldraw[black] (2,0) circle (2pt) node[anchor=north] {$b_2$};
\filldraw[black] (1,1.73) circle (2pt) node[anchor=east] {$b_3$};
\filldraw[black] (5,2.5) circle (0pt) node[anchor=west] {$b_1=(1-t)a_2+(1-t)a_3+(2t-1)a_1$,}; 
\filldraw[black] (5,1.9) circle (0pt) node[anchor=west] {$b_2=(1-t)a_1+(1-t)a_3+(2t-1)a_2$,}; 
\filldraw[black] (5,1.3) circle (0pt) node[anchor=west] {$b_3=(1-t)a_1+(1-t)a_2+(2t-1)a_3$.}; 
\end{tikzpicture}
\caption{$\lambda_n^t(a)\in \bc(V_1, M_3(\bc)^{\otimes n})$}
 \label{fig-iterosc}
\end{figure}
\\
The inequality $\osc \big(\l^{t}_n(a)\big)\leq t\osc (a)$ means
$\| \l^{(t)}_n(a)(e^+)-\l^{(t)}_n(a)(e^-) \| \leq t\osc (a)$ for any edge $e\in E_1$. 
By  the symmetric invariance of $\l^{(t)}_n$, it is enough to check such inequality  for just two edges:
the edge whose boundary values are $a_1$ and $b_3$ and the edge whose boundary values are $b_1$ and $b_3$. Indeed
\begin{align*}
\|a_1-b_3\|&=\|(1-t)(a_1-a_2)+(2t-1)(a_1-a_3)\|\leq t\max(\|a_1-a_2\|,\|a_1-a_3\|)\leq t\osc(a),\\
\|b_3-b_1\|&=\|(1-t)(a_1-a_3)+(2t-1)(a_3-a_1)\|\leq t \|a_1-a_3\|\leq t\osc(a),
\end{align*}
and the inequality is proved.

As for the second, we identify $a$ with $a\otimes I$, which can be seen as an element of $\bc(V_0\sqcup V_0\sqcup V_0,M_3(\bc)^{\otimes m})$. Also, $\lambda_n^{(t)}(a)\in \bc(V_1)\otimes M_3(\bc)\subset\bc(V_0\sqcup V_0\sqcup V_0,M_3(\bc)^{\otimes m})$ so that their difference may be described as in fig. \ref{fig-difference}. A direct comparison of fig. \ref{fig-iterosc} and fig. \ref{fig-difference} shows that 
 $\|a-\l^{(t)}_n(a)\|\leq\osc(\l_n^t(a)) \leq t\osc(a)$.
\begin{figure}[ht]
\begin{tikzpicture}
\draw[gray, thick] (0,0) -- (2,0);
\draw[gray, thick] (0,0) -- (1,1.73);
\draw[gray, thick] (2,0) -- (1,1.73);
\draw[gray, thick] (3,0) -- (5,0);
\draw[gray, thick] (3,0) -- (4,1.73);
\draw[gray, thick] (4,1.73) -- (5,0);
\draw[gray, thick] (1.5,2.6) -- (3.5,2.6);
\draw[gray, thick] (2.5,4.33) -- (1.5,2.6);
\draw[gray, thick] (2.5,4.33) -- (3.5,2.6);
\filldraw[black] (0,0) circle (2pt) node[anchor=east] {$0$};
\filldraw[black] (2.5,4.33) circle (2pt) node[anchor=south] {$0$};
\filldraw[black] (5,0) circle (2pt) node[anchor=west] {$0$};
\filldraw[black] (4,1.73) circle (2pt) node[anchor=west] {$a_3$-$b_1$};
\filldraw[black] (3.5,2.6) circle (2pt) node[anchor=west] {$a_2$-$b_1$};
\filldraw[black] (1.5,2.6) circle (2pt) node[anchor=east] {$a_2$-$b_3$};
\filldraw[black] (2,0) circle (2pt) node[anchor=north] {$a_1$-$b_2$};
\filldraw[black] (3,0) circle (2pt) node[anchor=north] {$a_3$-$b_2$};
\filldraw[black] (1,1.73) circle (2pt) node[anchor=east] {$a_1$-$b_3$};
\end{tikzpicture}
\caption{$a-\lambda_n^t(a)\in\bc(V_0\sqcup V_0\sqcup V_0,M_3(\bc)^{\otimes m})$}
 \label{fig-difference}
\end{figure}
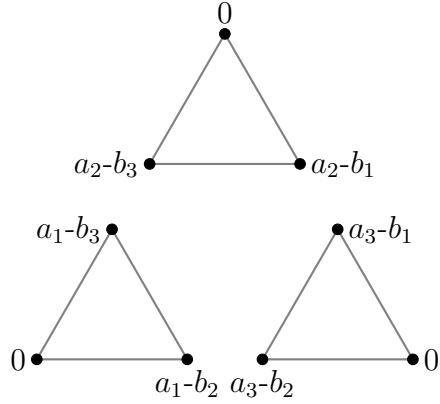
\end{proof}

In the following, we drop the superscript $t$ if not needed.
\begin{lem}\label{itenspi}
For all $n\geq 0$, $\l_n$ sends $\ca_n$ in $\ca_{n+1}$.
\end{lem}
\begin{proof} Check the property on each generator of $\ca_n$, as described in Lemma \ref{indicesBis}.
\end{proof}
We also define
$\l_{[n,n+k]}$ as the iterate $\l_{[n,n+k]}= \l_{n+k-1}\circ \cdots \circ \l_n$ from $M_3(\bc)^{\otimes n}\otimes \bc^3$ into $M_3(\bc)^{\otimes n}\otimes\bc(V_k)\subset M_3(\bc)^{\otimes n}\otimes( \bc^3)^{\otimes (k+1)}$. This is a completely positive contraction which sends $\ca_n$ into $\ca_{n+k}$.
\begin{prop}[Symmetric extension from $\ca_n$ into $\ca_{\infty}$]\label{EA2}
\item[$(1)$] For all  $n\in \bn$ and  $x\in M_3(\bc)^{\otimes n}\otimes \bc^3$, the sequence $\l_{[n,n+k]}(x)$ has a uniform limit in the UHF algebra as $k\to \infty$. This limit will be denoted $\l_{[n,\infty)}(x)$.
\item[$(2)$] $\l_{[n,\infty)}$ is a completely positive contraction which sends $\ca_n$ into $\ca_{\infty}$.
\item[$(3)$] For $b\in \ca_n$, $\l_{[n,\infty)}(b)$ is an extension of $b$ in the sense that $\rho_n\big(\l_{[n,\infty)}(b)\big)=b$.
\end{prop}
In other terms, every element of an approximating algebra extends naturally to an element of the noncommutative Gasket with the same norm.

\begin{proof}
$(1)$ Let us write $x=x^1\otimes \a^0_1+x^2\otimes \a^0_2 + x^3\otimes  \a^0_3$ the generic element of $M_3(\bc)^{\otimes n}\otimes \bc^3=M_3(\bc)^{\otimes n}\otimes \bc(V_0)$.
\\
By construction
$$\l_{[n,n+k]}(x)=x^1\otimes \l_{[0,k]}(\a^0_1)+x^2\otimes \l_{[0,k]}(\a^0_2)+x^3\otimes \l_{[0,k]}(\a^0_3),$$
where  $\l_{[0,k]}(\a^0_i)$ is the symmetric extension of the function $\a^0_i$ from $\bc(V_0)$ to $\bc(V_{k})$. 
By Proposition \ref{ext1} (3),
$$
\|\l_{[n,n+j]}(\a^0_i)-\l_{[n,n+j+1]}(\a^0_i)\|\leq t \osc(\l_{[n,n+j]}(\a^0_i))\leq t^{j+1},
$$
hence the sequence $\l_{[n,n+k]}(x)$ is a Cauchy sequence.
\\
$(2)$ By Lemma \ref{itenspi}, $\lim_n \l_{[n,n+k]}(x)\in\ca_\infty$ whenever $x$ belongs to $\ca_n$.
\\
$(3)$ Just notice that $\rho_0\circ \l_{[0,\infty)}(\a^0_i)=\a^0_i$.
\end{proof}

\begin{cor}
 Each morphism $\rho_m\,:\,\ca_{\infty}\to \ca_m$ is onto.

\end{cor}

Next Proposition shows that each element of $\ca_\infty$ appears in many ways as a limit of elements obtained through a symmetric extension process.

\begin{prop}\label{EA3} \-
Let $b\in \ca_{\infty}$.

\item[$(1)$] For any approximating sequence, $b_n\in \ca_n$, $\lim b_n=b$, we have
$$b=\lim_{n\to \infty} \l_{[n,\infty)}(b_n)\,.$$

\item[$(2)$] In particular\,:
$$b=\lim_{n\to \infty} \l_{[n,\infty)}(\rho_n(b))\,.$$

\end{prop}
\begin{proof}
$(1)$ Let us identify $b_n$ with $b_n\otimes I_3\in M_3(\bc)^{\otimes n}\otimes \bc^3\otimes \bc^3$, in order to get first $\l_{n+1}(b_n)=b_n\otimes I_3\otimes I_3$, then $\l_{[n+1,\infty)}(b_n)=b_n\in \UHF(3^\infty)$. As $\l_{[n+1,\infty)}$ is a contraction, we obtain
$$
\|\l_{[n+1,\infty)}(b_{n+1})-b_n\|\leq \|b_{n+1}-b_n\|\to 0\,, \quad n\to \infty,
$$
from which we deduce $\lim_n \l_{[n+1,\infty)}(b_{n+1})=\lim_n b_n=b$.

\item[$(2)$] is a consequence of $(1)$ above and Proposition \ref{restrizioni}.
\end{proof}

\begin{rem} Summarizing, we have selected a dense subspace of the noncommutative Gasket generated by explicit elements
$$
\l_{[n,\infty)}(\a^n_j),\ x\otimes \l_{[n,\infty)}(\b^{k}_j):j=1,2,3, k=1,\dots n, n\in\bn,x\in M_3(\bc)^{\otimes (n-k)}.
$$
\end{rem}

\begin{cor} \label{nuc}
The $C^*$-algebra $\ca_{\infty}$ is nuclear.
\end{cor}
\begin{proof}
The sequence of linear maps $b\to \l_{[n,\infty)}\circ{\rho_n(b)}$ provide an explicit approximation of the identity by completely positive contractions with finite rank.
\end{proof}

\section{Representations of $\ca_\infty$}
\subsection{The ideal $\ca_{\infty}^0$ of the elements in $\ca_{\infty}$ vanishing on $V_0$.}

\begin{dfn}
Denote by $\ca_{\infty}^0$ the kernel of the homomorphim $\rho_0\,:\,\ca_{\infty}\to \bc(V_0)$, and by  $\ca_{n}^0$ the kernel of the homomorphim $\rho_0|_{\ca_n}\,:\,\ca_n\to \bc(V_0)$.
%
%
\end{dfn}

All these are ideals in their respective $C^*$-algebras.

\begin{rem}
One checks easily that the kernel of the homomorphim $\rho_0|_{\bc(V_n)}\,:\, \bc(V_n) \to \bc(V_0)$  is the ideal of functions on $V_n$ vanishing on $V_0$, while the kernel of the homomorphim $\rho_0|_{C(K)}\,:\, C(K) \to \bc(V_0)$ is the space of continuous functions on the Gasket vanishing on $V_0$.
\end{rem}


The following Lemma is easy, but useful\;:

\begin{lem}\label{rho-m-maps}\-

\item[$(1)$] The restriction map $\rho_n$ sends $\ca_{n+1}^0$ in $\ca_{n}^0$.

\item[$(2)$] By iteration and passing to the limit, $\rho_n$ sends $\ca_{n+k}^0$ in $\ca_{n}^0$ and $\ca_{\infty}^0$ in $\ca_{n}^0$.

\end{lem}


\begin{prop}\label{kerrhom}\-

\item[$(1)$] For all $m,p \in \bn$, $M_3(\bc)^{\otimes m} \otimes \ca_p^0 \subset \ca_{m+p}^0$.

\item[$(2)$] The kernel of $\rho_m\,:\;\ca_{m+p}\to \ca_m$ is $M_3(\bc)^{\otimes m}\otimes \ca_p^0$.

\item[$(3)$] For all $m\in\bn$, $M_3(\bc)^{\otimes m}\otimes \ca_\infty^0 \subset \ca_\infty^0$.

\item[$(4)$] The kernel of $\rho_m\,:\; \ca_{\infty}\to \ca_m$ is $M_3(\bc)^{\otimes m}\otimes \ca_\infty^0$.

\end{prop}
\begin{proof}
$(1)$ $M_3(\bc)^{\otimes m}\otimes \ca_p^0$ is generated by elements of the form $x_m\otimes y_k\otimes \b_j^{p-k}$, $x_m\in M_3(\bc)^{\otimes m}$, $y_k\in M_3(\bc)^{\otimes k}$, so they are of the form $z_{m+k}\otimes  \b_j^{p-k}$, $z_{m+k}\in M_3(\bc)^{\otimes(m+k)}$, and therefore belong to $\ca_{m+p}^0$.\\

\item[$(2)$] The inclusion $M_3(\bc)^m\otimes \ca_p^0\subset \ker \rho_m$ is easy to prove. To show the opposite inclusion, we take an element in $\ca_{m+p}$
$$
b=\sum_{j=1}^3 \xi_j\a_j^{m+p}+\sum_{k=0}^{m+p-1}\sum_{j=1}^3 b_{j,k}\otimes\b_j^{m+p-k}\,,\;\xi_j\in \bc\,,\; b_{j,k}\in M_3(\bc)^{\otimes k}\,,
$$
and compute
$$
\rho_m(b)=\sum_{j=1}^3 \xi_j\a_j^{m}+\sum_{k=0}^{m-1}\sum_{j=1}^3 b_{j,k}\otimes\b_j^{m-k}\,,\;\xi_j\in \bc\,,\; b_{j,k}\in M_3(\bc)^{\otimes k}\,.
$$
If $\rho_m(b)=0$, then $\displaystyle
b=\sum_{k=m}^{m+p-1}\sum_{j=1}^3 b_{j,k}\otimes\b_j^{m+p-k}$, where $b_{j,k}\otimes\b_j^{m+p-k}$ is in $M_3(\bc)^k\otimes \b_j^{m+p-k}\subset M_3(\bc)^m\otimes M_3(\bc)^{k-m}\otimes \b_j^{m+p-k}\subset M_3(\bc)^m\otimes \ca_p^0$\,.

\item[$(3)$] For $b\in \ca_\infty^0$ and $x_m\in M_3(\bc)^{\otimes m}$,  proposition \ref{omegan} implies that  $x_m\otimes b=\lim_p\rho_{m+p}(x_m\otimes b)=\lim_p x_m\otimes \rho_p(b)$, where $\rho_p(b)\in \ca_p^0$. Because of $(1)$, $x_m\otimes \rho_p(b)\in \ca^0_{m+p}$, and, taking the limit,  $x_m\otimes b\in \ca_\infty^0$.

\item[$(4)$] We apply $(2)$: if $b\in \ca_{\infty}$, then $b\in \ker \rho_m$ if and only if, for any $p$, $\rho_{m+p}(b)\in \ker \rho_m$, that is, for any $p$, $\r_{m+p}(b)\in M_3(\bc)^{\otimes m}\otimes \ca_p^0$.
\end{proof}

\begin{cor}\label{corkerrhom} \-

\item[$(1)$] $M_3(\bc)^{\otimes m}\otimes \ca_\infty^0$ is a closed ideal. 

\item[$(2)$] The quotient $\ca_{\infty} \big/\big(M_3(\bc)^{\otimes m}\otimes \ca_\infty^0\big)$ is canonically isomorphic to $\ca_m$  (so it is finite dimensional).

\end{cor}

\begin{cor}
Let us denote by $C_0(K)$ the continuous functions on $K$ vanishing on the three boundary vertices of $K$. Then the elements of  the UHF algebra $\UHF(3^\infty)$ of the form
$$
x_m\otimes f, \quad x_m\in M_3(\bc)^{\otimes m},\, f\in C_0(K)
$$
belong to, and indeed generate,  the ideal $\ca_\infty^0$.
\end{cor}
\begin{proof}
The algebra  $M_3(\bc)^{\otimes m}\otimes C_0(K)$ is contained in $\ca^0_\infty$ by Corollary \ref{corkerrhom} (1). 
From Lemma \ref{indicesBis}, one deduces easily that $\ca_n^0$ is linearly generated by the $x\otimes \beta_j^k$, $1\leq k\leq n$, $j=1,2,3$, $x\in M_3(\bc)^{\otimes (n-k)}$.
Given $b\in \ca^0_\infty$ one has, by Lemma \ref{rho-m-maps}, $\rho_n(b)\in \ca_n^0$, hence $\rho_n(b)$ is a linear combination of $x\otimes \beta_j^k$.
Finally, by Proposition \ref{EA3} (2), $b=\lim_n \lambda_{[n,\infty)}(\rho_n(b))$, where  $\lambda_{[n,\infty)}(\rho_n(b)$ is a linear combination of elements of the form $x\otimes \lambda_{[k,\infty)}(\beta_j^k)$, each of them lying in some $M_3(\bc)^{\otimes (n-k)}\otimes C_0(K)$. 
\end{proof}
\begin{lem}\label{inter}
$\bigcap_m ker(\rho_m)=\{0\}$.
\end{lem}
\begin{proof}
If $b\in \bigcap_m ker(\rho_m)$, then $b=\lim_m\rho_m(b)=0$\,.
\end{proof}


\subsection{Representations orthogonal to every $\rho_m$.}

Each $\rho_m$ can be interpreted as a finite dimensional representation of $\ca_{\infty}$ in $\ca_m\subset M_3(\bc)^{\otimes m+1}=\cb\big((\bc^3)^{\otimes m+1}\big)$. This family of representations can be summarized into
$$\rho_\infty := \oplus_m \rho_m\,: \,\ca_{\infty}\to \oplus_m \cb\big((\bc^3)^{\otimes m+1}\big)\,.$$

\begin{prop}\label{disjunction}
\item[$(1)$]  Let $\pi$ be a representation of $\ca_{\infty}$ in some $\cb(H)$, which is orthogonal to every $\rho_m$, i.e. orthogonal to $\rho_\infty$ ({\rm orthogonal } means {\rm no intertwining operators}).
Then $\pi$ is faithful and extends to a representation $\overline \pi$ from $\UHF(3^\infty)$ in $\cb(H)$, with $\pi(\ca_{\infty})''=\overline \pi(\UHF(3^\infty))''$.
\item[$(2)$] If $\pi$ is a representation of the UHF such that $\pi_{|\ca_{\infty}}\perp \oplus_m \rho_m$, then $\pi(\ca_{\infty})''=\pi(\UHF(3^\infty))''$.
\end{prop}
\begin{proof}
$(1)$ Let us first notice that the UHF being a simple algebra, all its representations are faithful and isometric. Hence, if $\pi$ extends to $\UHF(3^\infty)$, it needs to be faithful.
\\
 With the assumptions of $(1)$, we make the following claim\,:\\
{\it For any $m\in \bn$, the restriction of $\pi$ to $M_3(\bc)^{\otimes m}\otimes \ca_\infty^0$ is non degenerate.}\\
To prove this claim, notice that, $M_3(\bc)^{\otimes m}\otimes \ca_\infty^0$ being an ideal (Corollary \ref{corkerrhom}), the space ${\big( (M_3(\bc)^{\otimes m}\otimes \ca_\infty^0)H\big)}^\perp$ is invariant for $\ca_{\infty}$, and thus is the space of a subrepresentation $\pi_0$ of $\pi$. On this space, $M_3(\bc)^{\otimes m}\otimes \ca_\infty^0$ acts by the null action. As $M_3(\bc)^{\otimes m}\otimes \ca_\infty^0= \ker(\rho_m)$ (Lemma \ref{kerrhom}), $\pi_0$ factorizes through $\rho_m$. As $\pi$ is orthogonal to $\rho_m$, this must be the null representation\,: ${\big( (M_3(\bc)^{\otimes m}\otimes \ca_\infty^0)H\big)}^\perp=\{0\}$. The claim is proved.
\\
We continue by fixing some $m$ and choosing a positive increasing approximate unit $(u_n)$ in the ideal $\ca_{\infty}^0$. Note that the $(I_3)^{\otimes m}\otimes u_n\in M_3(\bc)^{\otimes m}\otimes \ca_\infty^0\subset \ca_\infty^0$ also form  an approximate unit for $\ca_\infty^0$ for any $m$, so that, according to the claim above, $\lim_n \pi((I_3)^{\otimes m}\otimes u_n)=I_H$ (strong limit)\,. The map
$$
\pi_m\,:\,M_3(\bc)^{\otimes m}\to \pi(\ca_{\infty})'' \,, \quad \pi_m(x)=\lim_n \pi(x\otimes u_n) \,,
$$
is well defined (for $x$ positive, it is an increasing sequence in $\cb(H)$). It is a morphism of $*$-algebras, and it does not depend on the choice of an approximate unit. Choosing  $I_3\otimes u_n$ instead of $u_n$ shows that
$$
\pi_{m+1}(x\otimes I_3)=\lim_n \pi(x\otimes I_3 \otimes u_n)=\pi_m(x) \,, \quad x\in M_3(\bc)^{\otimes m} \,,
$$
i.e. that the restriction of $\pi_{m+1}$ to $M_3(\bc)^{\otimes m}$ is $\pi_m$.

Taking the limit on $m$, we get a representation $\overline \pi$ from $\UHF(3^\infty)$ into $\pi(\ca_{\infty})''$. It remains to show that $\overline \pi$ extends $\pi$, i.e. coincides with $\pi$ on $\ca_{\infty}$.

Fix $b\geq 0$ in $\ca_{\infty}$, and $\varepsilon >0$. For $m$ large enough, we have, in $\UHF(3^\infty)$, $\|b-\rho_m(b)\|\leq \varepsilon$. Choosing an approximate unit $(u_n)$ in $\ca_\infty^0$, we have the following sequence of inequalities in $\UHF(3^\infty)$:
\begin{equation*}\begin{split}
\rho_m(b)-\eps I \leq &b \leq \rho_m(b)+\eps I \\
\big(I_3^{\otimes m+1}\otimes u_n\big)\big(\rho_m(b)-\eps I\big) \leq \big(I_3^{\otimes m+1}\otimes u_n^{1/2}\big)\,&b\,\big(I_3^{\otimes m+1}\otimes u_n^{1/2}\big) \leq \big(I_3^{\otimes m+1}\otimes u_n\big)\big(\rho_m(b)+\eps I\big)\\
\rho_m(b)\otimes u_n -\eps\,I_3^{\otimes m+1}\otimes u_n \leq \big(I_3^{\otimes m+1}\otimes u_n^{1/2}\big)\,&b\,\big(I_3^{\otimes m+1}\otimes u_n^{1/2}\big)  \leq \rho_m(b)\otimes u_n +\eps\,I_3^{\otimes m+1}\otimes u_n
\end{split}\end{equation*}
Since the last double inequality involves only elements in $\ca_{\infty}$, we can apply $\pi$ to get the corresponding inequalities in $\cb(H)$, then let $n\to \infty$ and apply again $\rho_m(b)-\eps I \leq b \leq \rho_m(b)+\eps I$. This provides
\begin{align*}
\pi\big(\rho_m(b)\otimes u_n\big) -\eps\,\pi(I_3^{\otimes m+1}\otimes u_n) & \leq \pi\big(I_3^{\otimes m+1}\otimes u_n^{1/2}\big)\,\pi(b)\,\pi\big(I_3^{\otimes m+1}\otimes u_n^{1/2}\big) \\
& \leq \pi\big(\rho_m(b)\otimes u_n\big) +\eps\,\pi\big(I_3^{\otimes m+1}\otimes u_n\big) \\
\overline \pi(\rho_m(b))-\eps I_H & \leq \pi(b)\leq \overline \pi(\rho_m(b))+\eps I_H \\
\overline \pi (b)-2\eps I_H & \leq \pi(b)\leq \overline \pi(b)+2\eps I_H\,,
\end{align*}
which terminates the proof of $(1)$.

\item[$(2)$] According to the claim above, the restriction of $\pi$ to any $M_3(\bc)^{\otimes m}\otimes \ca_\infty^0$ is non degenerate. If $(u_n)$ is an approximate unit for $\ca_\infty^0$, we have $\lim_n\pi(I_3^{\otimes m}\otimes u_n)=I_H$ for the strong topology, which provides
$$
\pi(x_m)=\lim_n \pi(x_m)\pi(I_3^{\otimes m}\otimes u_n)=\lim_n\pi(x_m\otimes u_n) \,, \quad x_m\in M_3(\bc)^{\otimes m}\,.$$
This proves $\pi(M_3(\bc)^{\otimes m})\subset \pi(\ca_{\infty})''$, for any m.
\end{proof}

In the next subsection, we prove the existence of such representations and provide a criterion for them.

\smallskip

\subsection{A criterion of disjunction}

Let us start by noticing that any representation $\pi$ of $\ca_{\infty}$ in a  Hilbert space provides by restriction a representation of $C(K)$,  hence a probability measure $\mu_\pi$ on $K$ (defined up to equivalence) such that $\pi(C(K))'' \approx L^\infty(K,\mu_\pi)$.

\begin{prop}\label{pi-ort}
If $\mu_\pi(V_\infty)=0$, then $\pi\perp \oplus_m\rho_m$.
\end{prop}
\begin{proof}
Let $\pi_0=\pi_m\circ \rho_m$ be a representation of $\ca_{\infty}$ factorizing through $\rho_m$ ($\pi_m$ being any representation of $\ca_m$). Its restriction to $C(K)$ is of the form $C(K)\ni f \to \pi_m(f_{|V_m})$ and has a spectral measure $\mu_{\pi_0}$ supported by $V_m$.

If $\pi$ has a subrepresentation of the form $\pi_0=\pi_m\circ \rho_m$, then $\mu_\pi\geq \mu_{\pi_0}$ dominates a measure concentrated on $V_m$, which implies $\mu_\pi(V_m)\not=0$. By contradiction, we get the proposition.
\end{proof}

\begin{ex}\label{Powers}
Let $\omega$ be a diagonal faithful state on $M_3(\bc)$. Then the state $\omega^{\otimes \bn}$ on $\UHF(3^\infty)$, restricted to $\ca_{\infty}$, induces on $C(K)$ the selfsimilar measure with weights $(\omega_{11},\omega_{22},\omega_{33})$. This measure being diffuse, the associated representation of $\ca_{\infty}$ is disjoint from any $\rho_m$.

\noindent In particular, if $\tau$ is the restriction to $\ca_{\infty}$ of the canonical trace $\tau_\infty$ on $\UHF(3^\infty)$, we have
$$
\pi_\tau \perp \oplus_m\rho_m \text{ and } \pi_\tau(\ca_{\infty})''=\pi_{\tau_\infty}(\UHF(3^\infty))''\,.
$$
\end{ex}

\subsection{A decomposition theorem}\-
Proposition \ref{disjunction} provides a natural decomposition of any representation of $\ca_{\infty}$ into a discrete and a continuous part. Using the notation $\rho_\infty=\oplus_m\rho_m$, we have:

\begin{thm}\label{thm:decomposition}\-
Any representation $\pi$ of $\ca_{\infty}$ can be uniquely decomposed as the direct sum of two representations
$$
\pi=\pi^d\oplus \pi^c \,,
$$
with $\pi^d$ contained in  a multiple of $\rho_\infty$ (i.e. $\pi^d$ is weakly contained in $\rho_\infty$) and $\pi^c\perp \rho_\infty$.\\
Moreover, $\pi^d$ is a direct sum of finite dimensional representations, while $\pi^c$ has no finite dimensional summand and extends to a representation $\overline {\pi^c}$ of $\UHF(3^\infty)$ with $\pi^c(\ca_{\infty})''=\overline {\pi^c}(\UHF(3^\infty))''$.

\end{thm}

\begin{dfn}
We shall say that a representation $\pi$ of $\ca_{\infty}$ is discrete if it is contained in  a multiple of $\oplus_m \rho_m$, i.e. $\pi=\pi^d$ and $\pi^c=0$.

We shall say that a representation $\pi$ of $\ca_{\infty}$ is continuous if $\pi\perp\oplus_m \rho_m$, i.e. $\pi=\pi^c$ and $\pi^d=0$. Note that a continuous representation is faithful.
\end{dfn}

\subsection{The discrete spectrum of $\ca_{\infty}$.}\label{spettrodiscreto}\-

An irreducible representation of $\ca_{\infty}$ is either discrete or continuous, which means that, with obvious notations, the spectrum $\widehat{\ca_{\infty}}$ of $\ca_{\infty}$ is the disjoint union
$$
\widehat{\ca_{\infty}}=\big(\widehat{\ca_{\infty}}\big)^d\cup \big(\widehat{\ca_{\infty}}\big)^c
$$
of its discrete part and its continuous part.

Elements in  $\big(\widehat{\ca_{\infty}}\big)^d$ are of the form $\pi_m\circ \rho_m$, with $\pi_m$ an irreducible representation of $\ca_m$. The spectrum of $\ca_m$ has been investigated in Corollary \ref{multiplicativity}. With the notation of this Corollary, one notices easily that the representations $\xi_j\circ \rho_{m+1}$ and $\xi_j \circ \rho_m$ coincide, and that the representations $\eta_{k,j}\circ \rho_{m+1}$ and $\eta_{k,j}\circ \rho_m$ are the same for $k\leq m-1$. So that the discrete part $\big(\widehat{\ca_{\infty}}\big)^d$ of the spectrum can be described

\begin{itemize}
\item[$(1)$] either as the increasing union $\cup_m (\rho_m)_*\big(\widehat{\ca_m}\big)$

\item[$(2)$] or as the family gathering the $\xi_j\circ \rho_0$ (which are characters of $\ca_{\infty}$) and the $\eta_{m-1,j}\circ \rho_m$, $j=1,2,3$, $m\geq 1$.
\end{itemize}

In any case, $\big(\widehat{\ca_{\infty}}\big)^d$ is a countable set. By Lemma \ref{inter}, it is a dense subset of $\widehat{\ca_{\infty}}$.

\subsection{The continuous spectrum of $\ca_{\infty}$.}\-

Here, the challenge is to show that this continuous part is not empty. It will be a consequence of the following proposition\,:

\begin{prop}
Let $\pi$ be a representation of $\ca_{\infty}$ in some separable $\cb(H)$, and $\displaystyle\pi=\int^\oplus_X \pi_xd\mu(x)$ a disintegration of $\pi$ into a Hilbert integral of irreducible representations.

If $\pi$ is continous, then $\pi_x\in \big(\widehat{\ca_{\infty}}\big)^c$\,,  $x\in X$ $\mu$-a.s.
\end{prop}
\begin{proof}
By contradiction. Let $X^d$ be the set $\{x\in X\,|\,\pi_x\in \big(\widehat{\ca_{\infty}}\big)^d\}$, and suppose $\mu(X^d)\not=0$. As $\big(\widehat{\ca_{\infty}}\big)^d$ is countable, this provides some $m\in \bn$ and $\pi_m\in \widehat{\ca_m}$ such that the set $Y^d_m=\{x\in X\,|\,\pi_x\sim \pi_m\circ \rho_m\}$ has non zero measure. Hence $\pi_m\circ \rho_m$ appears as a subrepresentation of $\pi$.
\end{proof}

\begin{rem}
$\big(\widehat{\ca_{\infty}}\big)^c$ coincides with the set of classes of irrreducible representations of $\UHF(3^\infty)$, the restriction of which to $\ca_{\infty}$ remains irreducible.
\end{rem}

{
Because of  Powers representations in Example \ref{Powers}, and those in Example \ref{JLexRep} $(6)$, we can say that  $\big(\widehat{\ca_{\infty}}\big)^c$ is uncountable\,.
}

\subsection{Primitive ideals of $\ca_{\infty}$.}\label{primid}\-

Primitive ideals are the kernels of irreducible representations. For $\ca_{\infty}$, the list is easy to write\,: there is $\{0\}$ (the common kernel for all $\pi\in \big(\widehat{\ca_{\infty}}\big)^c$) and the countable family $\ker(\pi)$, $\pi \in \big(\widehat{\ca_{\infty}}\big)^d$.

\begin{prop}
The set of primitive ideals of $\ca_{\infty}$ is countable and can be enumerated this way, where $\x$ and $\eta$ were defined in \eqref{xi_and_eta}:
\begin{itemize}
\item[$(1)$] $\{0\}$

\item[$(2)$] $\ker(\xi_j\circ \rho_0)$, $j=1,2,3$

\item[$(3)$] $\ker(\eta_{m-1,j}\circ \rho_m)$, $j=1,2,3$, $m\geq 1$.
\end{itemize}
\end{prop}

\begin{rem}
The ideals of $\ca_{\infty}$ are obtained as all possible intersections of the primitive ideals above. The non-trivial ideals of $\ca_\infty$ are of the form
$\ker(\pi)$, $\pi\subset \oplus_m\rho_m$.

\end{rem}

\subsection{Tracial states on $\ca_{\infty}$}\label{traces}\-

With the notations in \eqref{xi_and_eta}, one can identify a family of obvious traces on $\ca_{\infty}$:

\begin{itemize}
\item[$(1)$] the trace $\tau_\infty$, restriction to $\ca_{\infty}$ of the natural and unique trace $\tau_3^{\otimes \bn}$ on $\UHF(3^\infty)$;

\item[$(2)$] the three characters $\chi_j=\xi_j\circ \rho_0$, $j=1,2,3$\,;

\item[$(3)$] the traces $\tau_{m,j}=\tau_3^{\otimes m-1}\circ \eta_{m-1,j}\circ \rho_m$ , $j=1,2,3$, $m\geq 1$\,.
\end{itemize}

As a matter of fact, they appear exactly as the set of extremal tracial states on $\ca_{\infty}$.

\begin{prop}  \label{JL_traces}
Any tracial state $\tau$ on $\ca_{\infty}$ can be written uniquely as
$$
\tau=\lambda_\infty \tau_\infty +\sum_j \lambda_j\chi_j+\sum_{j,m} \lambda_{m,j}\tau_{m,j}
$$
with $\lambda_\infty,\,\lambda_j,\,\lambda_{m,j}\geq 0$, $\lambda_\infty+\sum_j \lambda_j+\sum_{j,m}\lambda_{j,m}=1$.
\end{prop}
\begin{proof}
The representation $\pi_\tau$ (with cyclic vector $\xi_\tau$) decomposes as $\pi_\tau=\pi_\tau^d\oplus \pi_\tau^c$, while $\tau=\tau^d+\tau^c$, with $\tau^d(b)=<\xi_\tau,\pi_\tau^d(b)\xi_\tau>$ and $\tau^c(b)=<\xi_\tau,\pi^c(b)\xi_\tau>$.

$\tau^c$ extends to a finite trace on $\UHF(3^\infty)$, and thus is proportional to $\tau_\infty$.

$\pi^d$ decomposes as a sum of finite dimensional irreducible representations, which corresponds to a decomposition of $\tau^d$ in a sum or a series of tracial positive linear forms, each of them being proportional either to one of the $\chi_j$ or to one of the $\tau_{m,j}$.
\end{proof}


\begin{prop}
The following statements hold:

\item[$(1)$] the centre of $\ca_{\infty}$ is $\bc$,

\item[$(2)$] $C(K)$ is a maximal abelian subalgebra of $\ca_{\infty}$.
\end{prop}

\begin{proof}
$(1)$ If $a$ is an element in the center of $\ca_\infty$, $\pi_\tau(a)$ is an element of the UHF which commutes with $\pi_\tau(\ca_\infty)''=\pi_\tau(UHF(3^\infty))''$, hence an element in the center of $\pi_\tau(UHF(3^\infty))''$ which is a factor, hence a multiple of the unit.

\item[$(2)$] Suppose $b\in\ca_\infty\cap C(K)'$. For any $n\in\bn$,   $\rho_n(b)\in\ca_n\cap C(V_n)'$ by the properties of $\rho_n$. Proposition \ref{MASA1} then imply $\rho_n(b)\in C(V_n)$, hence  $b=\lim_n \rho_n(b)\in C(K)$.
\end{proof}

\begin{ex} [Representations of $\ca_\infty$] \label{JLexRep}

\item[$(1)$] Let $\omega$ be the state $(e_{11}^*)^{\otimes \bn}$ on $UHF(3^\infty)$. Clearly $\o$ is a pure state, and the representation $\pi_{\omega}$ of $\UHF(3^\infty)$ is irreducible. The restriction of $\omega$ to $C(K)$ is the Dirac measure of $11111\cdots$, which does not belong to $V_\infty$. Then, by Propositions \ref{disjunction} and \ref{pi-ort}, $\pi_\omega(\ca_\infty)''=\pi_\omega(UHF)''=\cb(H_\omega)$\,. Therefore $\omega_{|\ca_{\infty}}$ is a pure state, belonging to the continuous series.

\item[$(2)$] If $v=P(w)$ belongs to $K\backslash V_\infty$, then the restriction of the state $\omega=\otimes_n e_{w_nw_n}^*$ to $\ca_{\infty}$ is pure, and belongs to the continuous series. Its restriction to $C(K)$ is the Dirac measure $\delta_v$.

\item[$(3)$] $\omega=\big(\frac{1}{3}\sum_{1\leq i,j\leq 3}e_{ij}^*\big)^{\otimes \bn}$ is a pure state on $\UHF(3^\infty)$, and its restriction to $C(\Sigma_\infty)$ is the equidistributed Bernoulli measure, while its restriction to $C(K)$ is the symmetric self-similar measure (which is a diffuse measure). In examples $(1)$-$(3)$ we have $\pi_\omega(\ca_{\infty})''=\pi_\omega(\UHF(3^\infty))''=\cb(H_\omega)$.

\item[$(4)$] $\omega=(e_{22}^*)^{\otimes n}$ is a pure state on $\UHF(3^\infty)$, and its restriction to $\ca_{\infty}$ is the character $\chi_2$ of Proposition \ref{JL_traces}. Therefore, the representation $\pi_\omega$ of $\UHF(3^\infty)$ is infinite dimensional, and its restriction to $\ca_{\infty}$ contains a one-dimensional representation, that is $\pi_\omega |_{\ca_{\infty}}$ is not irreducible.

\item[$(5)$] The state $\omega_0=\frac{1}{2}\big( e_{11}^*+e_{13}^*+e_{31}^*+e_{33}^*\big)$ is pure on $M_3(\bc)$. Therefore
$$
\omega=\frac{1}{2}\big( e_{11}^*+e_{13}^*+e_{31}^*+e_{33}^*\big) \otimes \big(e_{22}^*\big)^{\otimes \bn^*}
$$
is a pure state of the $\UHF(3^\infty)$. Its restriction to $\ca_{\infty}$ factorizes through $\rho_0$, and it is equal to $\omega_0\circ \rho_0$. Since $\rho_0$ sends $\ca_{\infty}$ onto $\bc(V_0)=\bc^3$, and the restriction of $\omega_0$ to $\bc^3$ is the semi sum of two Dirac measures, we get
$$
\omega_{|\ca_{\infty}}=\frac{1}{2}\big(\delta_1+\delta_3\big)\circ \rho_0=\frac{1}{2}\big(\chi_1+\chi_3\big) \,,
$$
which is not a pure state.

\item[$(6)$] Any diagonal matrix $(\mu_i \delta_{ij})_{ij=1\cdots 3}$  $\mu_i\geq 0$, $\mu_1+\mu_2+\mu_3=1$, can be considered a probability measure on $\{1,2,3\}$, and can be extended to a pure state on $M_3(\bc)$, represented by the matrix $\big(\sqrt{\mu_i\mu_j}\big)$. If $(\mu^{(k)})_{k\in \bn}$ is a family of probability measures on $\{1,2,3\}$, and each of them is extended, as above, to a pure state $\omega^{(k)}$ on $M_3(\bc)$, and if the measure $\mu=\otimes_k\mu_k$ on $C(\Sigma_\infty)$ is diffuse, then $\mu$ extends to a pure state $\omega=\otimes_k\omega^{(k)}$ on $\UHF(3^\infty)$, and its restriction to $\ca_{\infty}$ is still a pure state, in the continuous series. Therefore, the continuous part of the spectrum of $\ca_\infty$ contains uncountably many elements.

\end{ex}


%
%
%
%
%
%
%
%

\section{The Dirichlet form}

\subsection{The classical case}

From the classical point of view, the construction of the Dirichlet form goes as follows. Given a quadratic form $\ce_0$ on $\bc(V_0)$, one may extend it by self-similarity to $\bc(V_1)$ as $\ce_1(f)=\sum_{j=1,2,3}\ce_0[f\circ w_j]$, and then project it back to $\bc(V_0)$ as $g\in\bc(V_0)\to\min\{\ce_1(f):f\in\bc(V_1),f|_{V_0}=g\}$. If such form on $\bc(V_0)$ is proportional to the original one, $\ce_0$ is called an eigenform, and the function $f$ realising the minimum is called the harmonic extension of $f$.
Note that one may keep extending  the energy by self-similarity step by step, via
\begin{equation}\label{sesiclas}
\ce_{n+1}[f]=\sum_{j=1,2,3}\ce_n[f\circ w_j],\quad f\in\bc(V_{n+1}),
\end{equation}
finally obtaining $\ce_n[f]=\sum_{\s\in\Sigma_n}\ce_0[ f\circ w_\s]$.

Given an eigenform, its extension by self-similarity to $\bc(V_n)$, and  the notion of harmonic extension of a function, one gets a closed Dirichlet form on a suitable dense sub-algebra of $C(K)$.

As is known, the symmetric eigenform for the gasket is $\ce_0[f]=\sum_{i\ne j}|f(v_i)-f(v_j)|^2$.
We now express the $n$-th combinatorial energy in a more algebraic way. First we associate with any $f\in\bc(V_n)$ an element $a_f\in\bc(V_0, M_3(\bc)^{\otimes n})$ setting $(a_f)_{\s\t}=\d_{\s\t}\cdot f\circ w_\s$.
Then,
$$
\tr(|a_f(v_i)-a_f(v_j)|^2)
=\sum_{\s,\t\in\Sigma_n}|(a_f)_{\s\t}(v_i)-(a_f)_{\s\t}(v_j)|^2
=\sum_{\s\in\Sigma_n}|f\circ w_\s(v_i)-f\circ w_\s(v_j)|^2.
$$
As a consequence,
the $n$-th combinatorial energy satisfies the following equation
\begin{align}
\ce_n[f]
&= \sum_{\s\in\Sigma_n}\ce_0[f\circ w_\s]
= \sum_{\s\in\Sigma_n}\sum_{i\ne j=1,2,3}|f\circ w_\s(v_i)-f\circ w_\s(v_j)|^2\label{tren}\\
&=\sum_{i\ne j=1,2,3}\sum_{\s\in\Sigma_n}|f\circ w_\s(v_i)-f\circ w_\s(v_j)|^2
= \sum_{i\ne j=1,2,3}\tr|a_f(v_i)-a_f(v_j)|^2.\notag
\end{align}
Equation \eqref{sesiclas} may also be reformulated in more algebraic terms.
Let us consider the linear map $e_{ij}^*\otimes id_n:M_3(\bc)^{\otimes (n+1)}
\to M_3(\bc)^{\otimes n}$. When $m\in M_3(\bc)^{\otimes (n+1)}$, $\a,\b\in\Sigma_n=\{1,2,3\}^{\times\, n}$, $\s=i\cdot\a$, $\t=j\cdot \b$, it satisfies
$$
\big((e_{ij}^*\otimes id_n)m\big)_{\a\b}=m_{\s\t}.
$$
Then, with the same symbols, for $f\in\bc(V_{n+1})$, we get
\begin{align*}
\big( (e_{ij}^*\otimes id_n)\circ a_f\big)_{\a\b}
&=(a_f)_{\s\t}
=\d_{\s\t}\cdot f\circ w_\s
=\d_{\a\b}\d_{ij}\cdot f\circ w_i\circ w_\a
=\d_{ij}\cdot\big(a_{f\circ w_i}\big)_{\a\b}.
\end{align*}
As a consequence,
\begin{align*}
\ce_n[f\circ w_i]
&=\sum_{p\ne q}\tr |a_{f\circ w_i}(v_q)-a_{f\circ w_i}(v_p)|^2\\
&=\sum_{j=1,2,3}\sum_{p\ne q}\sum_{\a\b}
|\d_{ij}\cdot(a_{f\circ w_i})_{\a\b}(v_p)-\d_{ij}\cdot(a_{f\circ w_i})_{\a\b}(v_q)|^2\\
&=\sum_{j=1,2,3}\sum_{p\ne q}\sum_{\a\b}
|\big( ( e_{ij}^*\otimes id_n)a_f\big)_{\a\b}(v_p)
-\big( ( e_{ij}^*\otimes id_n)a_f\big)_{\a\b}(v_q)|^2\\
&=\sum_{j=1,2,3}\sum_{p\ne q}
\tr|( e_{ij}^*\otimes id_n)a_f(v_p)-(e_{ij}^*\otimes id_n)a_f(v_q)|^2.
\end{align*}
It follows that the self-similarity equation \eqref{sesiclas} takes the form
\begin{equation}\label{sesiquan}
\sum_{p\ne q}\tr|a_f(v_p)-a_f(v_q)|^2=
\sum_{i,j=1,2,3}\sum_{p\ne q}
\tr|( e_{ij}^*\otimes id_n)a_f(v_p)-(e_{ij}^*\otimes id_n)a_f(v_q)|^2.
\end{equation}

%

\subsubsection{The harmonic extension}
The harmonic extension $\f$ from $\bc(V_0)$ to $\bc(V_1)$  relative to the energy $\ce_0$ is given by the symmetric extension with parameter $3/5$ described in \eqref{t-ext},
\begin{equation} \label{pizero}
\f(\a^0_j)=\a^1_j+\frac25\b^1_{j}+\frac25\b^1_{j+1}+\frac15\b^1_{j-1} .
\end{equation}
Then the harmonic extension from $\bc(V_n)$ to $\bc(V_{n+1})$ can be written as the restriction to $\bc(V_n)$ of the completely positive contraction
$$
\f_n=id_n\otimes \f:  (\bc^3 )^{\otimes(n+1)}=(\bc^3 )^{\otimes n}\otimes \bc(V_0) \longrightarrow (\bc^3 )^{\otimes n}\otimes \bc(V_1)\subset (\bc^3 )^{\otimes(n+2)}\,,
$$
The eigenform property is expressed, for $f\in\bc(V_n)$, by
\begin{equation}\label{selfsimform}
\ce_{n+1}[\f_n(f)]=\frac35\ce_n[f].
\end{equation}
By the minimizing property of the harmonic extension, for any $f\in C(K)$ the sequence $\left(\frac53\right)^n\ce_n[\rho_n(f)]$ is non decreasing, so that there exists 
$\ce[f]=\lim_n\left(\frac53\right)^n\ce_n[\rho_n(f)]$.

As above, harmonic extensions may be composed, so as to associate with any element $f\in\bc(V_{n})$ an element $\f_{[n,m]}(f)\in \bc(V_m)$, $m>n$, and also an element $\f_{[n,\infty]}(f)\in C(K)$, and, by the extension property, for $m>n$,
\begin{equation}\label{extprop}
\r_m\circ\f_{[n,\infty]}(f)=\r_m\circ\f_{[m,\infty]}(f)\circ\f_{[n,m]}(f)=\f_{[n,m]}(f).
\end{equation}
Then \eqref{selfsimform} and \eqref{extprop} give
$$
\left(\frac53\right)^m\ce_m[\rho_m\circ\f_{[n,\infty]}(f)]
=\left(\frac53\right)^m\ce_m[\f_{[n,m]}(f)]
=\left(\frac53\right)^{n}\ce_{n}[f],\qquad f\in\bc(V_n),\quad m\geq n,
$$
namely the sequence defining $\ce[\f_{[n,\infty]}(f)]$ is eventually constant, so that $\f_{[n,\infty]}(f)$ has finite energy. Such elements, for $n\in\bn$ form a dense subset of $C(K)$, therefore $\ce$ is a closed densely defined  Dirichlet form on $C(K)$.


\subsection{The non commutative case: the harmonic structure}\-


%
The Dirichlet form $\ce_n$ on $\ca_n$ will be the natural amplification of the form $\ce_0$ on $\ca_0=\bc(V_0)$:
\begin{dfn}\label{formaEn}
The  Dirichlet form $\ce_n$ on $\ca_n \subset M_3(\bc)^{\otimes n}\otimes \bc(V_0)$ is defined as
\begin{equation*}
\ce_n[b]
=\sum_{\a,\b\in \Sigma_n} \ce_0[b_{\alpha,\beta}]
=\sum_{\a,\b\in \Sigma_n} \sum_{i\ne j} |b_{\a\b}(v_i)-b_{\a\b}(v_j)|^2
=\sum_{i\ne j} \tr |b(v_i)-b(v_j)|^2\,.
\end{equation*}
\end{dfn}
It is not difficult to show that the definition above imply a self-similarity equation analogous to that of \eqref{sesiquan}.
\begin{prop}[Self-similarity]\label{autosim}
$$\ce_{n+1}[b]=\sum_{i,j=1..3} \ce_n\big[(e_{ij}^*\otimes id_n)\,b\big]\,,\;b\in \ca_{n+1}\,.$$
\end{prop}
Conversely, the self-similarity relation above recovers, by induction,  the whole sequence $\ce_n$ from its starting point $\ce_0$. We now prove that such sequence of forms provides a harmonic structure.
\begin{prop}\label{arm}{\bf Harmonic structure}
\item[$(1)$] For any $b\in \ca_n$,
$\displaystyle\inf\big\{\ce_{n+1}[a]:a\in\ca_{n+1},\rho_n(a)=b\big\}=\frac35\ce_n[b]$.
\item[$(2)$] The lower bound is reached on one and only one element $(id_n\otimes \f)(b)$ of $\ca_{n+1}$, where $\f$ is the harmonic extension from $\bc(V_0)$ to $\bc(V_1)$, see \eqref{pizero}.
\end{prop}
\begin{proof}
We noticed $\ca_{n+1}\subset M_3(\bc)^{\otimes n}\otimes \bc(V_1)$, so that the generic element $a$ of $\ca_{n+1}$ can be identified with an element of $\bc\big(V_1\to M_3(\bc)^{\otimes n}\big)$.

For fixed $b=\ca_n$,
we have $\rho_n(a)=b$ if and only if, for each $\alpha,\beta\in\Sigma_n$, $\rho_0(a_{\alpha\beta})=b_{\alpha,\beta}$. Which means $b_{\alpha\beta}={a_{\alpha\beta}}|_{V_0}$\, $\forall \alpha,\beta\in \Sigma_n$.
\\
For such $a$ extending $b$, the classical result of \cite{Kiga} implies $\displaystyle \ce_1(c_{\alpha,\beta})\geq \frac{3}{5}\ce_0(b_{\alpha\beta})$, with equality if and only if $c_{\alpha\beta}$ is the harmonic extension of $b_{\alpha\beta}$.
\\
Summing up, we get that, whenever $\rho_n(a) =b$, we have $\ce_{n+1}[a]\geq \frac{3}{5}\ce_n[b]$, with equality if and only if $a=(id_n\otimes \f)b$. \end{proof}

Summarizing the results from Section \ref{sec:symext}, the elements of the form
$\f_{[n,\infty)}(a), a\in\ca_n,n\in\bn$ form
 a dense subspace of finite energy elements in the noncommutative Gasket.

\subsection{The Dirichlet form}
As an immediate consequence of the two previous subsections, we get the following Proposition:
\begin{prop}\label{energia}{\bf Energy form on $\ca_\infty$}
\item[$(1)$] For $b\in \ca_{\infty}$, the sequence $\displaystyle \frac{5^m}{3^m}\ce_m[\rho_m(b)]$ is nondecreasing. Hence its limit
 $$\ce_\infty[b]=\lim_m \displaystyle \frac{5^m}{3^m}\,\ce_m[\rho_m(b)]$$ exists in $[0,+\infty]$.
 \item[$(2)$] $\ce_\infty$ is a densely defined quadratic form on $\ca_{\infty}$, closed for the uniform topology.
\\
In other terms, its domain
 $$\mathcal B=\{b\in \ca_{\infty}\,|\,\ce_\infty[b]<+\infty\}$$
 is dense in $\ca_\infty$ and a complete space for the norm
 $\left(||b||_{\ca_{\infty}}^2+\ce_\infty[b]\right)^{1/2}$\,.
\end{prop}
\begin{proof}
Everything is a consequence of the properties above. For the density of $\mathcal B$ in $\ca_\infty$, observe that for $b=\varphi_{[n,\infty)}(b_n)$, the sequence $ \frac{5^m}{3^m}\,\ce_m[\rho_m(b)]$ is stationary.
\end{proof}

\begin{lem}\label{norm1}\-

\item[$(1)$] For $b\in \mathcal B$, the following inequality holds true
   $$\displaystyle ||\rho_{n+1}(b)-\rho_n(b)\otimes I_{n+2} ||^2 \leq \frac{3^{n+1}}{5^{n+1}}\,\ce_\infty[b]
$$ for the operator norm in $M_3(\bc)^{\otimes (n+2)}$.

\item[$(2)$] There exists  $C$ independent of $n$ such that, $\forall b\in \mathcal B$, $$\displaystyle ||\rho_{n+1}(b)-\rho_0(b)\otimes I_{[2,n+2]} ||^2 \leq C\,\ce_\infty[b]
$$
and
$$\displaystyle ||b-\rho_0(b)\otimes I_{[2,\infty)} ||^2 \leq C\,\ce_\infty[b]\,.
$$
\end{lem}
\begin{proof}
$(2)$ is a mere consequence of $(1)$.

To prove $(1)$, observe first that $\displaystyle \ce_{n+1}[\rho_{n+1}(b)]\leq \frac{3^{n+1}}{5^{n+1}}\ce_\infty[b]$ (prop. \ref{energia}).

Then write $\rho_{n+1}(b)\in M_3(\bc)^{\otimes(n+1)}\otimes \bc^3$ as
$$\rho_{n+1}(b)=x\otimes e_{11}+y\otimes e_{22}+z\otimes e_{33}$$
with $x,y,z\in M_3(\bc)^{\otimes(n+1)}$ and $y=\big(id_{n+1}\otimes e_{22}^*\big)(\rho_{n+1}(b))=\rho_n(b)$.

Compute $\ce_m$ according to Definition \ref{formaEn}\,:
\begin{equation*}\begin{split}
 \ce_{n+1}[\rho_{n+1}(b)]&=\sum_{\b,\a\in \{1,2,3\}^{n+1}}|x_{\b,\a}-y_{\b,\a}|^2+|y_{\b,\a}-z_{\b,\a}|^2+|z_{\b,\a}-x_{\b,\a}|^2\\
 &=||x-y||^2_{L^2(Tr)}+||y-z||^2_{L^2(Tr)}+||z-x||^2_{L^2(Tr)}\\
 &=||x-\rho_n(b)||^2_{L^2(Tr)}+||\rho_n(b)-z||^2_{L^2(Tr)}+||z-x||^2_{L^2(Tr)}\,.
\end{split}\end{equation*}
As the operator norm is dominated by the $L^2$-norm (with respect to the nonnormalized trace), we get
\begin{equation*}
\begin{split}
||\rho_{n+1}(b)-\rho_n(b)\otimes Id||^2_{op} &=||x\otimes e_{11}+y\otimes e_{22}+z\otimes e_{33} -\rho_n(b)\otimes (e_{11}+e_{22}+e_{33} )||^2\\
&=\max(||x-\rho_n(b)||^2,||z-\rho_n(b)||^2)\\
&\leq \max(||x-\rho_n(b)||^2_{L^2(Tr)},||z-\rho_n(b)||^2_{L^2(Tr)})\\
&\leq \ce_{n+1}[\rho_{n+1}(b)]\\
&\leq  \frac{3^{n+1}}{5^{n+1}}\ce_\infty[b]\,.
\end{split}
\end{equation*}
\end{proof}

\begin{cor}\label{norm2}
For $b\in \mathcal B\,\cap\,  \ca_\infty^0$
$$
||b||^2\leq C\,\ce_\infty[b]\,.
$$
\end{cor}
\begin{proof}
Apply $(2)$ of lemma \ref{norm1} with $\rho_0(b)=0$.
\end{proof}

Notice that $\mathcal B\cap \ca_\infty^0$ has codimension $3$ in $\mathcal B$.

\medskip

In the sequel, $\tau$ will denote the normalized trace on $\UHF(3^\infty)$, as well as its  restrictions to the subalgebras $M_3(\bc)^{\otimes n}$, $\bc(V_n)$, $\ca_\infty$, $C(\Sigma_\infty)$ and  $C(K)$. On $C(\Sigma_\infty)$, it is the equidistributed Bernoulli measure; on $C(K)$, it coincides with the symmetric self-similar measure.

$E_\tau$ will denote the conditional expectation from $\UHF(3^\infty)$ onto $C(\Sigma_\infty)$. Its restriction to $M_3(\bc)^{\otimes n}$ is the restriction of a matrix to its diagonal\,: $E_\tau(X)_{\alpha,\beta}=\delta_{\alpha,\beta} X_{\alpha,\beta}$, $\alpha,\beta\in \Sigma_n$. One checks easily that $E_\tau(\ca_n)=\bc(V_n)$ (check on each generator of $\ca_n$), and consequently that $E_\t$ sends $\ca_\infty$ onto $C(K)$.

Moreover, it is obvious from the Definition
\ref{formaEn} of $\ce_n$ that $\ce_n[E_\tau(b)]\leq \ce_n[b]$, and consequently $E_\tau(\mathcal B)\subset \mathcal B$, with $\ce_\infty[E_\tau(b)]\leq \ce_\infty[b]$.

\begin{prop}
[Noncommutative Sobolev inequality]\label{Feller}
There exists a constant $C'$ such that
$$||b||_{\ca_{\infty}}^2 \leq C'\,\big(\ce_\infty[b]+\tau(b^*b)\big)\,,\,b\in \mathcal B\,.$$

\end{prop}
\begin{proof}

Let us fix $b\in \mathcal B$ and notice that

$(a)$   $E_\tau(b)$ lies in $C(K)$, hence $(e_{ij}^*\otimes id)E_\tau(b)=0$, for $i\not=j$\,;

$(b)$  $b-E_\tau(b)$ has vanishing diagonal components, i.e. $(e_{ii}^*\otimes id)\big(b-E_\tau(b)\big)=0$ for $i=1,2,3$\,;

$(c)$  the quadratic form $\ce_\infty$ being autosimilar, one can apply Lemma  \ref{autosim} which, passing to the limit and polarizing, provides for the associated sesquilinear form\,:
$$\ce_\infty\big(E_\tau(b)\,,\,b-E_\tau(b)\big)=\frac{5}{3}\sum_{ij} \ce_\infty\big(\underbrace{e_{ij}^*\otimes id_{2,\infty}E_\tau(b)}_{=0 \text{ for }i\not=j}, \underbrace{e_{ij}^*\otimes id_{2,\infty}(b-E_\tau(b)}_{=0 \text{ for }i=j})
\big)=0\,;$$
from which we conclude\,: $\ce_\infty[b]=\ce_\infty[E_\tau(b)]+\ce_\infty[b-E_\tau(b)]$\,;

$(d)$  the projection $\rho_0(b)$ depends  only on diagonal components of $b$, hence $\rho_0(E_\tau(b))=\rho_0(b)$;

$(e)$  $b-E_\tau(b)\in \ca_\infty^0$ satisfies the assumptions of Corollary \ref{norm2}\,;

$(f)$  the result we seek to prove is true in the classical case \cite{Kiga}:  there exists $C_0$ such that
\begin{equation}\label{clasineq}
||b||_{C(K)}^2 \leq C_0\,\big(\ce_\infty[b]+\tau(b^*b)\big)\,,\,b\in \mathcal B\cap C(K)\,.
\end{equation}
By \eqref{clasineq} and Corollary \ref{norm2}, we get for any $b\in \mathcal B$
\begin{equation*}
\begin{split}
||b||^2&\leq 2\big(||E_\tau(b)||^2+||b-E_\tau(b)||^2) \\
&\leq 2C_0\big(\ce_\infty[E_\tau(b)] + \tau(E_\tau(b)^*E_\tau(b))\big)+2C \ce_\infty[b-E_\tau(b)]\\
&\leq 2\max(C,C_0)\big(\ce_\infty[b]+\tau(b^*b)\big)\,.
\end{split}
\end{equation*}
 \end{proof}
Let us summarize the previous results as a theorem:
\begin{thm}
\item[$(1)$] $\cb$ is an algebra and is complete for the norm $\big(||b||_{|L^2(\tau)}^2+\ce[b]\big)^{1/2}$.
\item[$(2)$] $\ce$ is a symmetric Dirichlet form on $L^2(\ca_{\infty},\tau)$.
\item[$(3)$] The associated Dirichlet algebra and Dirichlet space coincide and are equal to $\cb$.
\end{thm}
\begin{proof}
$(2)$ and $(3)$ are immediate consequences of $(1)$. For $(1)$, it is enough to observe that, according to Proposition
 \ref{energia}, the norms $\big(||b||_{|L^2(\tau)}^2+\ce[b]\big)^{1/2}$ and $\big(||b||_{\ca_{\infty}}^2+\ce[b]\big)^{1/2}$ are equivalent. Lemma \ref{Feller} provides the result.
\end{proof}

\section{A spectral triple for the noncommutative gasket}
Spectral triples $(H,D,\pi)$ allow to introduce the analogue of a Riemannian structure on a differentiable manifold $M$ where, for example, $H=L^2(\Lambda^* (M))$ is the Hilbert space of square integrable differential forms, $D=d+d^*$ is the Dirac operator and $\pi$ the action of the algebra $C(M)$ on $H$ by pointwise multiplication.\\
We now consider the discrete spectral triple on the continuous functions on the gasket defined in \cite{GuIs10} . We refer to \cite{GuIs16} for further details. 
Set $(\ch_n,D_n, \pi_n)$ as follows: $\ch_n=\ell^2(E_n)$, where $E_n$ is the set of oriented edges of level $n$, $D_n=2^n F_n$, where $F_n$ changes the orientation of edges, $\pi_n(f)e=f(e^+)$, with $e^+$ the target of $e$. Then we define
$(\ch,D, \pi)$ as $\ch=\oplus_n\ch_n$, $D=\oplus_n D_n$, $\pi=\oplus_n\pi_n$.

In view of the description of $C(K)$ given in Theorem \ref{Kinfinity}, we can reformulate the triples $(\ch_n,D_n, \pi_n)$ as follows. Set $\ch_n=(\bc^3)^{\otimes n}\otimes E$, with $E=\{x\in M_3(\bc):x_{jj}=0, j=1,2,3.\}$, with the scalar product on $E$ given by $(x,y)=\tr(x^*y)$. Any edge in $E_n$ can be written as $w_\s e$, where $|\s|=n$ and $e\in  E_0$. The map $w_\s$ is identified with the element $e_{\s_1\s_1}\otimes e_{\s_2\s_2} \otimes\dots\otimes e_{\s_n\s_n}\in(\bc^3)^{\otimes n}$, while denoting by $\ell_1,\ell_2,\ell_3$ three consecutive edges in the boundary of the triangle oriented clockwise, and by $\tilde \ell_1,\tilde \ell_2,\tilde \ell_3$ the edges with the opposite orientations, the identification is as follows: $\ell_1\leftrightarrow e_{21},\ell_2\leftrightarrow e_{32},\ell_3\leftrightarrow e_{13}$, the change in the orientation being given by the transposition of matrices.
In this way the operator $F_n$ on elements $x\otimes m\in (\bc^3)^{\otimes n}\otimes E$ acts as $F_n (x\otimes m)=x\otimes m^T$. It is a matter of computation to show that the action of $C(K)$ on $\ch_n$ is recovered by the position $\pi_n(a)z=\rho_n(a)z$, $z\in \ch_n$, where, since $\bc(V_n)$ is contained in $M_3(\bc)^{\otimes n}\otimes\bc^3\subset M_3(\bc)^{\otimes n}\otimes M_3(\bc)$, the action $(a\otimes b)(x\otimes m)$ is intended as $ax\otimes bm$, with the first product being the natural action of a matrix on a vector, and the second being the matrix multiplication; indeed the product between a diagonal matrix and a matrix in $E$ is still in $E$. We have proved that
\begin{prop}
The spectral triples $(\ch_n,D_n,\pi_n)$  on $C(K)$ can be equivalently described as follows: $\ch_n=(\bc^3)^{\otimes n}\otimes E$, $D_n=2^nF_n$, with $F_n (x\otimes m)=x\otimes m^T$, 
$\pi_n(a)z=\rho_n(a)z$, $z\in \ch_n$.
\end{prop}
We can now define spectral triples on $\ca_\infty$. We keep the spaces $\ch_n$ and  $\ch$ and the operators $F_n$, $D_n$ and $D$  as above. The representations $\pi_n$ are formally the same as above, namely $\pi_n(a)x=\rho_n(a)x$, but $\rho_n$ is now extended to $\ca_\infty$.

In order to show that there is a dense $^*$-algebra of elements in $\ca_\infty$ for which the commutator with $D$ is bounded, we consider the symmetric extension corresponding to $t=1/2$ defined in  \eqref{t-ext}. We denote such extension by the letter $\th$, and call it the affine extension.

\begin{prop}\label{QMS}
The seminorm $a\in\ca_\infty\mapsto L(a)=\|[D,a]\|$ is a Lip-norm in the sense of Rieffel \cite{Rief}, namely $(\ca_\infty,L)$ is a quantum metric space.
More precisely, the set $\cl=\{a\in\ca_\infty:L(a)<\infty\}$ is a dense $*-$algebra of $\ca_\infty$, and $(\ca_\infty)_{1,1}=\{a\in\ca_\infty:\|a\|\leq1,L(a)\leq1\}$ is relatively compact in norm.
\end{prop}
\begin{proof}
Let us choose $a\in\ca_n$. We first observe that $\|[F_n,a]\|=\osc(a)$. 
Indeed, since $a\in M_3(\bc)^{\otimes n}\otimes\bc^3$, it can be described as a diagonal matrix $\diag (a_1,a_2,a_3)$, with entries $a_i\in M_3(\bc)^{\otimes n}$. Analogously, since $\ch_n=(\bc^3)^{\otimes n}\otimes E$, a vector $x\in \ch_n$ can be described as a matrix $(x_{ij})_{i,j=1,2,3}$ with entries $x_{ij}\in (\bc^3)^{\otimes n}$. By a simple computation, $((a-F_n a F_n)x)_{ij}=(a_i-a_j)x_{ij}$, from which
$\|(a-F_n a F_n)x\|^2=\sum_{i,j}\|(a_i-a_j)x_{ij}\|^2\leq \max_{ij}\|a_i-a_j\|^2\|x\|^2=\osc(a)^2\|x\|^2$. We then have $\|a-F_n a F_n\|\leq \osc(a)$, and this bound is clearly attained. Finally, since $F_n$ is a unitary operator, $\|[F_n,a]|=\|F_n(a-F_naF_n)\|=\osc(a)$.
As a consequence, making use of Proposition \ref{ext1} (3),
$$
\|[D_{n+p},\th_{[n,n+p]}(a)]\|=2^{n+p}\|[F,\th_{[n,n+p]}(a)]\|
=2^{n+p}\osc(\th_{[n,n+p]}(a))\leq 2^{n}\osc(a)=\|[D_{n},a]\|.
$$
Moreover, as $a\in\ca_{n}$, $\osc(\r_{n-1}(a))\leq2\osc(a)$, therefore, for $k\leq n$,
$$
\|[D_k,\rho_k(a)]\|=2^k\osc(\rho_k(a))\leq2^k2^{n-k}\osc(a)=\|[D_n,a]\|.
$$
This shows that
\begin{align*}
\|[D,\th_{[n,\infty)}(a)]\|
&=\sup_{k\in\bn}\|[D_k,\rho_k\circ \th_{[n,\infty)}(a)]\|\\
&=\max\left(\max_{k=1,\dots n}\|[D_k,\rho_k(a)]\|,\sup_{k>n}\|[D_k,\rho_k\circ \th_{[n,\infty)}(a)]\|\right)=\|[D_n,a]\|,
\end{align*}
therefore the elements $\{\th_{[n,\infty)}(a):a\in\ca_n\}$ have bounded commutator with $D$. Since such set is dense by Proposition \ref{EA3}, the seminorm $L(a)$ is densely defined.
\\
We now observe that, for $a\in\ca_\infty$, $\|\r_{n+1}(a)-\r_n(a)\|\leq\osc(\r_{n+1}(a))$, which implies that
\begin{align*}
\|\r_{n+k}(a)-\r_n(a)\|
&\leq\sum_{j=0}^{k-1}\|\r_{n+j+1}(a)-\r_{n+j}(a)\|
\leq\sum_{j=0}^{k-1}\osc(\r_{n+j+1}(a))\\
&=\sum_{j=0}^{k-1}2^{-(n+j+1)}\|[D_{n+j+1},\r_{n+j+1}(a)]\|
\leq 2^{-n}L(a),
\end{align*}
therefore, sending $k\to\infty$, $\|a-\r_n(a)\|\leq 2^{-n}L(a)$.
\\
In order to show that $L$ is a Lip-norm, it is enough to build $\eps$-nets for the set $(\ca_\infty)_{1,1}$.
We recall that both norm and Lip-norm are preserved by the restriction $\r_n$ and the extension $\th_{[n,\infty)}$.
First we choose $n$ such that $2^{-n}\leq\frac\eps3$, then we pick an $\frac\eps3$-net $X_n$ in $(\ca_n)_{1,1}=\{b\in\ca_n:\|b\|\leq1,\|[D_n,b]\|\leq 1 \}$; we claim that $\{\th_{[n,\infty)}(b):b\in X_n\}$ is an $\eps$-net  for the set $(\ca_\infty)_{1,1}$. Indeed, given $a\in(\ca_\infty)_{1,1}$, we can find $x_a\in X_n$ such that $\|\r_n(a)-x_a\|<\eps/3$. Then,
$$
\|a-\th_{[n,\infty)}(x_a)\|
\leq \| a-\r_n(a)\| + \|\r_n(a)-x_a\| + \|x_a-\th_{[n,\infty)}(x_a)\| <\eps.
$$
\end{proof}

\begin{thm}
The triple $(\ch,D,\pi)$ on $\ca_\infty$ given by $\ch=\oplus_n\ch_n$, $D=\oplus_n D_n$, $\pi = \oplus_n\pi_n$, is a spectral triple with metric dimension $d=\frac{\log 3}{\log 2}$. The formula $\Res_{s=d}\tr(a|D|^{-s})$ produces the restriction of the normalized trace on the $UHF(3^\infty)$ up to a positive constant, the formula
$$
\Res_{s=\delta}\tr\big(|[D,a]|^2|D|^{-s}\big);\ \delta=2-\frac{\log5/3}{\log2},
$$
reproduces the Dirichlet form up to a positive factor, on the domain of the Dirichlet form.
\end{thm}
\begin{proof}
The bounded commutator property follows by Proposition \ref{QMS}.
The compact resolvent property and the computation of the metric dimension depend only on the spectrum of $D$, so they  follow by the analogous results for the triple on $C(K)$.

We now study the trace formula. By definition of the harmonic extension $\f_0:\ca_0\to\ca_1\subset\bc^3\otimes\bc^3$ we get $\tr\f_0(x)=3\tr x$, hence for $\f_n:\ca_n \to \ca_{n+1}$ we get $\tr\f_n(x)=3\tr x$ and for $\f_{n,n+k}:\ca_n \to \ca_{n+k}$ we get $\tr\f_{n,n+k}(x)=3^k\tr x$. We now fix $n\geq0$, $a\in\ca_n$ and consider $b=\f_{n,\infty}(a)\in\ca_\infty$. We have $\tr\r_{n+k}(b)= \tr\r_{n+k}(\f_{n,\infty}(a))= \tr\r_{n+k}(\f_{n+k,\infty}(\f_{n,n+k}(a)))=\tr \f_{n,n+k}(a)=3^k\tr a$. Then,
\begin{align*}
\Res_{s=d}\tr(b|D|^{-s})
&=\Res_{s=d}\left(\sum_{j=0}^{n-1}2^{-sj}\tr\r_j(b)+\sum_{k=0}^\infty2^{-s(n+k)} \tr\r_{n+k}(b)\right)\\
&=\Res_{s=d}2^{-sn}\tr a\sum_{k=1}^\infty e^{(\log3-s\,\log2)k}
=\frac{3^{-n}}{\log2}\tr a
\end{align*}
Therefore, for the normalized trace $\t$ on $\UHF(3^\infty)$, which coincides with $3^{n+1}\tr$ on $\ca_n$, and for any $c\in\ca_\infty$, $a=\r_n(c)$, we get
$$
\Res_{s=d}\tr(\f_{n,\infty}(\r_n(c))|D|^{-s})=\frac3{\log2}\tau(\r_n(c)).
$$
Since, for any $c\in\ca_\infty$, $\f_{n,\infty}(\r_n(c))\to c$ in $\ca_\infty$ and $\r_n(c)\to c$ in $\UHF(3^\infty)$, and both functionals are bounded, we get the thesis.

As for the energy, recall that $\ca_n \subset M_3(\bc)^{\otimes n}\otimes\bc^3$, $\ch_n=(\bc^3)^{\otimes n}\otimes E$. Then, a generic element of $\ca_n$ may be written as $\sum_{i=1,2,3}a_i\otimes e_{ii}$. Then
$$
[D_n,\sum_{i=1,2,3}a_i\otimes e_{ii}]x\otimes e=
2^n\sum_{i=1,2,3}a_i x\otimes \left(e^Te_{ii}-e_{ii}e^T\right)
$$
hence
\begin{align*}
\tr\left(\big|[D_n,\sum_{i=1,2,3}a_i\otimes e_{ii}]\big|^2\right)
&=2^{2n}\sum_{|\s|=n,
\,j\ne k}
\| \sum_{i=1,2,3}
a_i \s\otimes \left(e_{kj}e_{ii}-e_{ii}e_{kj}\right)\|^2\\
&=2^{2n}\sum_{|\s|=n,
\,j\ne k}
\| (a_j-a_k)\s\otimes e_{kj}\|^2
\\&
=2^{2n}\sum_{j\ne k}
\tr(| (a_j-a_k)|^2)=2^{2n}\ce_n[\sum_{i=1,2,3}a_i\otimes e_{ii}].
\end{align*}
Finally, when $a$ has finite energy,
\begin{align*}
\Res_{s=\delta} \tr\big( |[D,\pi(a)]|^2|D|^{-s} \big)
& = \Res_{s=\delta} \sum_n e^{((2-s) \log2-\log\frac53)n} \left(\frac53\right)^n\ce_n[\r_n(a)]
= \frac1{\log2} \ce_\infty[a].
\end{align*}
\end{proof}

\begin{ack}
This work has been partially supported by GNAMPA, MIUR, the European Networks ``Quantum Spaces - Noncommutative Geometry" HPRN-CT-2002-00280, and ``Quantum Probability and Applications to Physics, Information and Biology'', LYSM Laboratoire Ypatia Science Mathematiques-Laboratorio Ypatia Scienze Matematiche and  the ERC Advanced Grant 227458 OACFT ``Operator Algebras and Conformal Field Theory".
D. G. and T. I. acknowledge the MIUR Excellence Department Project awarded to the Department of Mathematics, University of Rome Tor Vergata, CUP E83C18000100006 and the University of Rome Tor Vergata funding scheme ``Beyond Borders'', CUP E84I19002200005.
\end{ack}



\begin{thebibliography}{99}

\bibitem{Ahlfors} L. V. Ahlfors. {\it Lectures on quasiconformal mappings}, Van Nostrand, Princeton, 1966.

\bibitem{ADTW} S.~Akiyama, G.~Dorfer, J.M..~Thuswaldner, R.~Winkler. {\it On the fundamental group of the Sierpinski gasket}, Topol. Appl. {\bf 156} (2009),  1655-1672.

\bibitem{AlGuPoSc} S.~Albeverio, D.~Guido, A.~Ponosov, S.~Scarlatti. {\it Singular traces and compact operators}, J. Funct. Anal. {\bf 137} (1996),  281-302.

\bibitem{BaJu} S. Baaj, P. Julg. {\it Th\'eorie bivariante de Kasparov et op\'erateurs non born\'es dans les C$^*$-modules hilbertiens}, C. R. Acad. Sci. Paris S\'er. I Math. 296 (1983), no. 21, 875-878.

\bibitem{Bar} M.~T. Barlow.  {\it Diffusions on fractals}.  L.N.M.   {\bf 1690},  Springer-Verlag, New York, 1999, 1-112.

\bibitem{BarKi} M.T. Barlow, J.~Kigami, {\it Localized eigenfunctions of the Laplacian on p.c.f. self-similar sets}, J. London Math. Soc. (2) 56 (1997), 320-332.

\bibitem{BarPer} M.T. Barlow, E.A. Perkins {\it Brownian motion on the Sierpinski gasket}, Probab. Theory Related Fields 79 (1989), 542-624.

\bibitem{BMR} J.V. Bellissard, M. Marcolli, K. Reihani. {\it Dynamical systems on spectral metric spaces}, (2010), arXiv:1008.4617 [math.OA].

\bibitem{PeBe} J. V. Bellissard, J. C. Pearson. {\it Noncommutative Riemannian geometry and diffusion on ultrametric Cantor sets}, J. Noncommut. Geom. 3 (2009), no. 3, 447-480.

\bibitem{BST} O. Ben-Bassat, R.S. Strichartz, A. Teplyaev {\it What is not in the domain of the Laplacian on the Sierpinski gasket type fractals} J. of Funct. Anal. {\bf 166} (1999), no. 2, 197-217.

\bibitem{BCR} C. Berg, J.P.R. Christensen, P. Ressel. {\it  Harmonic analysis on semigroups. Theory of positive definite and related functions}. Springer-Verlag, New York, 1984.

\bibitem{BD} A. Beurling, J. Deny. {\it Dirichlet Spaces.}, Proc. Nat. Acad. Sci., {\bf 45} (1959), 208-215.

\bibitem{CRSS} A. L. Carey, A. Rennie, A. Sedaev, F. A. Sukochev {\it The Dixmier trace and asymptotics of zeta functions} Journal of Functional Analysis {\bf 249} (2007) 253-283

\bibitem{CPS} A.L. Carey, J. Phillips, F. A. Sukochev, {\it Spectral Flow and Dixmier Traces}. Advances in Math. {\bf 173} (2003), 68-113.

\bibitem{CIL} E.~Christensen, C.~Ivan, M.~L.~Lapidus, {\it Dirac operators and spectral triples for some fractal sets built on curves}. Adv. Math. 217 (2008), no. 1, 42-78.

\bibitem{CIS} E. Christensen, C. Ivan, E. Schrohe, {\it Spectral triples and the geometry of fractals}. arXiv:1002.3081v2

\bibitem{CGIS1} F. Cipriani, D. Guido, T. Isola, J.-L. Sauvageot, {\it Differential 1-forms, their Integrals and Potential Theory on the Sierpinski Gasket},  Adv. Math. 239 (2013), 128-163.  

\bibitem{CGIS2} F. Cipriani, D. Guido, T. Isola, J.-L. Sauvageot, {\it Spectral triples for the Sierpinski gasket}, J. Funct. Anal. 266 (2014), no. 8, 4809-4869.

\bibitem{CiSa} F.~Cipriani, J.L.~Sauvageot. {\it Derivations as square roots of Dirichlet forms}, Journal of Functional Analysis, {\bf 201} (2003), 78-120.

\bibitem{Cip} F.~ Cipriani.  {\it Dirichlet forms on Noncommutative Spaces}.  L.N.M.   "Quantum Potential Theory "{\bf 1954},  U. Franz - M Sch\"urmann eds. Springer-Verlag, New York, 2008, 161-272.

\bibitem{Co} A. Connes. {\it Noncommutative Geometry}. Academic Press, 1994.

\bibitem{CoMa} A.Connes, M.Marcolli, {\it A walk in the noncommutative garden}. An invitation to noncommutative geometry, 1-128, World Sci. Publ., Hackensack, NJ, 2008.

\bibitem{Dix} Dixmier, Jacques
{\it Formes lin\'eaires sur un anneau d'op\'erateurs}. Bulletin de la Soci\'et\'e Math\'ematique de France, 81 (1953),  9-39

\bibitem{Erdelyi} A.~Erd\'elyi et al., {\it Higher Transcendental Functions, Vol I},  McGraw-Hill, New York, 1953

\bibitem{FK} T. Fack, H. Kosaki. {\it Generalized s-numbers of $\t$-measurable operators}. Pacific J. Math., {\bf 123} (1986), 269.


\bibitem{FOT} M. Fukushima, Y. Oshima, M. Takeda. {\it Dirichlet Forms and Symmetric Markov Processes}, W. de Gruyter, Berlin-New York, 1994.

\bibitem{FS} M. Fukushima, T.~Shima.   {\it On a spectral analysis for the Sierpinski gasket}.  Potential Analysis, {\bf 1} (1992),  1-35.

\bibitem{GeVi4}
I.~M.~Gel'fand,  N.~Ya.~Vilenkin. {\it Generalized functions. Vol. 4. Applications of harmonic analysis.}  Academic Press, New York-London, 1964.

\bibitem{Gold} S. Goldstein.  {\it Random walk and diffusions on fractals, Percolation theory and ergodic theory of infinite particle systems}.  H. Kesten Ed., IMA Math. Appl. Vol. 8 Springer (1987),  121-129.

 \bibitem{GuIs} D.~Guido, T.~Isola.   {\it Singular traces on semifinite von Neumann algebras}.  Journ. Funct.  Analysis, {\bf 134} (1995),  451-485.

 \bibitem{GuIs9} D.~Guido, T.~Isola.  {\it Dimensions and singular
 traces for spectral triples, with applications to fractals}.  Journ.
 Funct.  Analysis, {\bf 203}, (2003) 362-400.

 \bibitem{GuIs10} D. Guido, T. Isola.  {\it Dimensions and spectral
 triples for fractals in $\br^{N}$}, Advances in Operator Algebras and Mathematical Physics; Proceedings of the Conference held in Sinaia, Romania, June 2003, F. Boca, O. Bratteli, R. Longo H. Siedentop Eds., Theta Series in Advanced Mathematics, Bucharest 2005.


\bibitem{GuIs16} D. Guido, T. Isola,  \emph{Spectral triples for nested fractals}, J. Noncommut. Geom., 11 (2017), 1413-1436


\bibitem{HiRoBook} N.~Higson, J.~Roe. {\it Analytic $K$-homology}.
 Oxford University Press, Oxford, 2000.

 \bibitem{Hino}
 M.~Hino, {\it On singularity of energy measures on self-similar sets}.
Probab. Theory Related Fields {\bf 132} (2005),  265-290.

 \bibitem{HiNa}
 M. Hino, K. Nakahara, {\it On singularity of energy measures on self-similar sets II}, Bull. London Math. Soc. {\bf 38} (2006),  1019-1032.

\bibitem{Hutch}
J.E. Hutchinson,
{\it Fractals and self-similarity},
Indiana Univ. Math. J., {\bf 30} (1981),  713-747.

 \bibitem{IRT}
    M. Ionescu, L. G. Rogers, A. Teplyaev,
    {\it Derivations and Dirichlet forms on fractals},
 arXiv:1106.1450.

\bibitem{Jo} A.~Jonsson. {\it A trace theorem for the Dirichlet form on the Sierpinski gasket},  Math. Z.,  {\bf 250} (2005),  599-609.

\bibitem{JuSa} A. Julien, J.  Savinien. {\it Transverse Laplacians for substitution tilings}, Comm. Math. Phys. 301 (2011),  285-318

\bibitem{Kiga} J.~Kigami,  {\it Analysis on fractals}, Cambridge Tracts in Mathematics, 143, Cambridge University Press, Cambridge, 2001.

\bibitem{KiLa1} J.~Kigami, M.~L.~Lapidus. {\it Weyl's problem for the spectral distribution of Laplacians on p.c.f. self-similar fractals}, Comm. Math. Phys. 158 (1993),  no. 1, 93-125

\bibitem{KiLa2} J.~Kigami, M.~L.~Lapidus. {\it Self-similarity of volume measures for Laplacians on p.c.f. self-similar fractals}, Comm. Math. Phys. 217 (2001), no. 1, 165-180

\bibitem{Ku1} S. Kusuoka {\it A diffusion process on a fractal}. Probabilistic methods in mathematical physics (Katata/Kyoto, 1985), 251-274, Academic Press, Boston, MA, 1987.

\bibitem{Ku} S. Kusuoka, {\it Dirichlet forms on fractals and products of random matrices}, Publ. Res. Inst. Math. Sci., {\bf 25}  (1989), 659-680.

 \bibitem{LvF} M.~L.~Lapidus, M.~van Frankenhuijsen. {\it Fractal geometry, complex dimensions and zeta functions. Geometry and spectra of fractal strings}. Springer Monographs in Mathematics. Springer, New York, 2006.

   \bibitem{Lewin} L.~Lewin.  {\it Polylogarithms and associated functions}.   North Holland, New York, 1981.

\bibitem{LPS}
S. Lord, D. Potapov, F. Sukochev,
{\it Measures from Dixmier traces and zeta functions}.
J. Funct. Anal. 259 (2010), no. 8, 1915-1949.


\bibitem{PW} C.~Plaut, J. Wilkins. {\it Discrete homotopies and the fundamental group},   Adv. Math. 232 (2013), 271-294.

\bibitem{R} R.~Rammal. {\it Spectrum of harmonic oscillations on fractals},  J. Physique, {\bf 45}  (1984), 191-206.

\bibitem{RT} R.~Rammal, G.~Toulouse. {\it Random walks on fractal structures and percolation clusters}, J. Phys. Lett., {\bf 44} (1983), L13-L22.

\bibitem{Rief} M.~A. Rieffel.  {\it Gromov-Hausdorff distance for quantum metric spaces}.  Mem.  Amer.  Math.  Soc.   {\bf 168} (2004), no.  796, 1-65.
 
\bibitem{SZSY} A. Skye, Z. Conn, R.S. Strichartz, H. Yu,  {\it Hodge-de Rham theory on fractal graphs and fractals}.  Commun. Pure Appl. Anal. 13 (2014), no. 2, 903-928.
 
\bibitem{Sierpinski} W.~Sierpinski.  {\it Sur une courbe dont tout point est un point de ramification}. C. R. A. S. {\bf 160} (1915), 302-30.

\bibitem{Varga} J.V.~Varga.  {\it Traces on irregular ideals}.  Proc. Amer. Math. Soc. {\bf 107}  (1989),  715-723.


\end{thebibliography}
\end{document}